\documentclass[reqno]{amsart}

\usepackage[utf8]{inputenc}
\usepackage[T1]{fontenc}
\usepackage[british]{babel,isodate}
\cleanlookdateon
\usepackage{lmodern}
\usepackage{enumerate}
\usepackage{physics}
\usepackage{ stmaryrd }
\usepackage{tikz}
\usepackage{tikz-cd}
\usetikzlibrary{arrows}
\usepackage{cancel}
\usetikzlibrary{babel}
\usepackage{lscape}
\usepackage{url}

\usepackage{hyperref}
\hypersetup{linktocpage,colorlinks=true,citecolor=purple,linkcolor=blue,urlcolor=blue,filecolor=black}

\usepackage{float}
\usepackage{graphicx, import}
\usepackage{mathtools}
\usepackage{dsfont}
\usepackage{pinlabel}
\usepackage{yhmath}

\usepackage{xargs}                      
\usepackage{xcolor} 
\definecolor{OliveGreen}{rgb}{0,0.6,0}
 
\usepackage[colorinlistoftodos, prependcaption,textsize=tiny]{todonotes}
\newcommandx{\unsure}[2][1=]{\todo[linecolor=red,backgroundcolor=red!25,bordercolor=red,#1]{#2}}
\newcommandx{\change}[2][1=]{\todo[linecolor=blue,backgroundcolor=blue!25,bordercolor=blue,#1]{#2}}
\newcommandx{\info}[2][1=]{\todo[linecolor=OliveGreen,backgroundcolor=OliveGreen!25,bordercolor=OliveGreen,#1]{#2}}

\newcommand{\centre}[1]{\begin{array}{c} #1 \end{array}}

\usepackage{stackengine}

\usepackage{amsmath,amsthm,amssymb, amsfonts, amscd}
\usepackage{rotating,subcaption}
\usepackage{url}
\usepackage{graphicx,bm}
\usepackage{mathtools}
\usepackage[mathscr]{euscript}
\allowdisplaybreaks
\usepackage{lipsum}
\usepackage{bm}
\usepackage[capitalize]{cleveref} 

\usepackage{xpatch}
\makeatletter
\AtBeginDocument{\xpatchcmd{\@thm}{\thm@headpunct{.}}{\thm@headpunct{}}{}{}}
\makeatother

\newtheorem{theorem}{Theorem}[section]
\newtheorem{proposition}[theorem]{Proposition}
\newtheorem{lemma}[theorem]{Lemma}
\newtheorem{corollary}[theorem]{Corollary}

\theoremstyle{definition}
\newtheorem{definition}[theorem]{Definition}
\newtheorem{example}[theorem]{Example}

\theoremstyle{remark}

\newtheorem{remark}[theorem]{Remark}
\newtheorem{notation}[theorem]{Notation}

\numberwithin{equation}{section}

\renewcommand{\to}{\longrightarrow}

\usepackage{accents}
\newcommand{\uhat}{\underaccent{\check}}

\makeatletter
\newcommand{\cupr@tip}{\text{\raisebox{-0.1ex}{$\m@th\hat{}$}}}
\newcommand{\cupr}{\mathbin{\cup\cupr@}}

\newcommand{\cupr@}{%
  \mathchoice
  {\mkern-1.35mu\cupr@tip}
  {\mkern-1.35mu\cupr@tip}
  {\mkern-1.55mu\cupr@tip}
  {\mkern-1.875mu\cupr@tip}
}
\makeatother

\makeatletter
\newcommand{\capr@tip}{\text{\raisebox{0.47ex}{$\m@th\uhat{}$}}}
\newcommand{\capr}{\mathbin{\capr@\cap}}

\newcommand{\capr@}{%
  \mathchoice
  {\mkern11.6mu\capr@tip\mkern-11.6mu}
  {\mkern11.4mu\capr@tip\mkern-11.4mu}
  {\mkern11.1mu\capr@tip\mkern-11.1mu}
  {\mkern10.2mu\capr@tip\mkern-10.2mu}
}
\makeatother

\makeatletter
\newcommand{\capl@tip}{\text{\raisebox{0.47ex}{$\m@th\uhat{}$}}}
\newcommand{\capl}{\mathbin{\capl@\cap}}

\newcommand{\capl@}{%
  \mathchoice
  {\mkern2.1mu\capl@tip\mkern-2.1mu}
  {\mkern2.1mu\capl@tip\mkern-2.1mu}
  {\mkern2.3mu\capl@tip\mkern-2.3mu}
  {\mkern2.1mu\capl@tip\mkern-2.1mu}
}
\makeatother

\makeatletter
\newcommand{\cupl@tip}{\text{\raisebox{-0.1ex}{$\m@th\hat{}$}}}
\newcommand{\cupl}{\mathbin{\cupl@\cup}}

\newcommand{\cupl@}{%
  \mathchoice
  {\mkern1.35mu\cupl@tip\mkern-1.35mu}
  {\mkern1.35mu\cupl@tip\mkern-1.35mu}
  {\mkern1.55mu\cupl@tip\mkern-1.55mu}
  {\mkern1.875mu\cupl@tip\mkern-1.875mu}
}
\makeatother

\DeclareFontFamily{U}{mathx}{}
\DeclareFontShape{U}{mathx}{m}{n}{ <-> mathx10 }{}
\DeclareSymbolFont{mathx}{U}{mathx}{m}{n}
\DeclareFontSubstitution{U}{mathx}{m}{n}
\DeclareMathAccent{\widecheck}{0}{mathx}{"71}

\begin{document}

\title{The Large-Color Expansion Derived from the Universal Invariant}
\date{\today}

\author{Boudewijn Bosch}
\address{Bernouilli Institute, University of Groningen, Nijenborgh 9, 9747 AG, Groningen, The Netherlands}
\email{\href{mailto:b.j.bosch@rug.nl}{b.j.bosch@rug.nl}}




\begin{abstract}
  The colored Jones polynomial associated to a knot admits an expansion of knot invariants known as the large-color expansion or Melvin--Morton--Rozansky expansion. We will show how this expansion can be derived from the universal invariant arising from a Hopf algebra $\mathbb{D}$, as introduced by Bar-Natan and Van der Veen. We utilize a Mathematica implementation to compute the universal invariant $\mathbf{Z}_{\mathbb{D}}(\mathcal{K})$ up to a certain order for a given knot $\mathcal{K}$, allowing for experimental verification of our theoretical results.
\end{abstract}

\subjclass{57K14, 16T05, 17B37}


\maketitle



\section{Introduction}

The colored Jones polynomial $J_{\mathcal{K}}^n$ is an extension of the Jones polynomial \cite{Jones1985API}, a knot invariant in three-dimensional space. It depends on a positive integer $n$, which corresponds to the dimension of an irreducible representation of the quantum group $U_q(\mathfrak{sl}_2)$. The colored Jones polynomial is an example of a so-called \textit{quantum invariant}, which typically arise from the representation theory of quantum groups \cite{Jimbo1985AqdifferenceAO,kassel}

The Alexander polynomial is another invariant of knots in three-dimensional space \cite{Alexander1928TopologicalIO}. It is a polynomial $\Delta_{\mathcal{K}}(T)$ in one variable $T$ that has clear roots in the topology of the knot $\mathcal{K}$.

Melvin and Morton conjectured a relationship between the colored Jones polynomial $J^n_{\mathcal{K}}$ and Alexander polynomial $\Delta_{\mathcal{K}}$. This was subsequently proven by Bar-Natan and Garoufalidis \cite{bar1996melvin}. Rozansky later showed that the colored Jones polynomial admits an expansion whose zeroth-order term is equal to the reciprocal of the Alexander polynomial \cite{rozansky1997higher,rozansky1998universalr}. This expansion is known as the \textit{large-color expansion}, \textit{rational expansion}, \textit{loop expansion} or \textit{Melvin--Morton--Rozansky expansion} (see Theorem~\ref{thm:rozoverbay}). In her PhD thesis, Overbay computed explicit expressions of the first- and second-order terms ($P_1^{\mathcal{K}}$ and $P_2^{\mathcal{K}}$ resp.) for knots up to 10 crossings \cite{overbay2013perturbative}.

      In recent years, there has been a resurgence of interest in the large-color expansion, due to the development of the Gukov and Manolescu series $F_K(x,q)$, for which an expansion in terms of the large-color expansion has been conjectured \cite{park2020large,gukov2021two}. Moreover, the Melvin--Morton expansion has been used to study the  similarly defined \textit{unified invariant} \cite{martel2024unified}.

Recently, Bar-Natan and Van der Veen studied the Lawrence's universal invariant \cite{lawrenceuniversal} for a particular choice of ribbon Hopf algebra $\mathbb{D}$. This knot invariant $\mathbf{Z}_{\mathbb{D}}$ dominates the universal quantum $\mathfrak{sl}_2$ invariant and hence all colored Jones polynomials \cite{bar2021perturbed}. A fundamental part of their construction was to consider an algebra $\mathbb{D}$ over $\mathbb{Q}[\epsilon] \llbracket h \rrbracket$, instead of the usual $\mathbb{Q} \llbracket h \rrbracket$ (see Definition~\ref{def:D}). The new variable $\epsilon$ does not play the same role as $h$ in the quantized enveloping algebras of Drinfeld--Jimbo type. One significant distinction is that $\epsilon$ is absent from the definition of the co-product.

By considering generating functions in terms of perturbed Gaussians, they were able to expand $\mathbf{Z}_{\mathbb{D}}(\mathcal{K})$ in orders of $\epsilon$ with terms $\rho_{i,j}^{\mathcal{K}}$ $(i,j \geq 0)$, q.v.~Theorem~\ref{thm:barnatanveen}. This expansion closely resembles the large-color expansion of the colored Jones polynomial. Because the knot invariant $\mathbf{Z}_{\mathbb{D}}(\mathcal{K})$ originates from the universal invariant --- as opposed to the more common Reshetikhin--Turaev invariant --- its topological interpretation is tractable due to its good behavior under tangle operations, strand doubling and reversal. Furthermore, a Mathematica implementation has been developed to compute $\mathbf{Z}_{\mathbb{D}}(\mathcal{K})$ efficiently.

Based on both the theory of quantum invariants, the Kontsevich invariant and experiments, the invariants $\rho^{\mathcal{K}}_{k,0}$ were conjectured to be equivalent to the higher-order knot invariants of the large-color expansion as introduced by Rozansky \cite{thesisbecerra}. In this paper, we will prove the following:
\begin{theorem}
The polynomials $\rho^{\mathcal{K}}_{1,0}$ and $P^{\mathcal{K}}_1$ are equal.
\end{theorem}
Moreover, we show that the large-color expansion can be obtained from $\mathbf{Z}_{\mathbb{D}}(\mathcal{K})$ for any order. Using the Mathematica implementation to compute $\mathbf{Z}_{\mathbb{D}}(\mathcal{K})$ up to a specified order of $\epsilon$, we provide experimental verification of this result.

\subsection{Organization of the paper}
In Section~\ref{sec:univinv} of this paper, we recall the definition of the universal invariant associated to rotational tangle diagrams. We introduce the concept of ``twisting'' the universal $R$-matrix, which allows for different choices of universal $R$-matrices. It is proven that the universal invariant of tangles varies through conjugation based on whether the standard $R$-matrix or its twisted counterpart is chosen.

In Section~\ref{sec:perturbedknot} we provide a brief review of $U_h(\mathfrak{sl}_2)$, together with an elaboration of its Verma modules. We proceed by showing that the universal $R$-matrix derived from the Drinfeld double $\mathfrak{D}_{\mathfrak{sl}_2}$ can be ``twisted'' to the universal $R$-matrix of $U_h(\mathfrak{sl}_2)$. By establishing an isomorphism between $\mathfrak{D}_{\mathfrak{sl}_2}$ and the algebra $\mathbb{D}$, an explicit relation between $\rho_{1,0}^{\mathcal{K}}$ and $P_1^{\mathcal{K}}$ can then be inferred.

In the final Section~\ref{sec:exver}, the Mathematica program by Bar-Natan and Van der Veen is used to experimentally verify the relation between the knot invariant $\mathbf{Z}_{\mathbb{D}}(\mathcal{K})$ and the large-color expansion of the colored Jones polynomial.

\subsection*{Acknowledgments} The author would like to thank Roland van der Veen and Jorge Becerra for many helpful discussions and suggestions on the content of this paper.

\section{The universal invariant and twisting}
\label{sec:univinv}

In this section, we will examine oriented and framed tangles which will be represented as rotational, mainly following \cite{becerra2024bar}. Additionally, we will define the universal invariant and discuss the operation of ``twisting.'' The universal invariant is shown to remain unchanged modulo conjugation regardless of whether the standard universal $R$-matrix or its twisted counterpart is chosen.
\subsection{Tangles}

A \textit{(oriented) tangle} is a compact oriented 1-manifold properly embedded in $I^{3}$ such that the boundary of the embedded 1-manifold consists of distinct points in $\{0\} \times I \times \{0,1\}$.  A tangle is said to be \textit{framed} when it is equipped with a non-singular normal vector field that equals $(0,-1,0)$ at all the endpoints. Two (framed) tangles are referred to as \textit{isotopic} if they can be transformed into one another through an isotopy of $I^{3}$ while preserving their boundary points.

Let $\mathrm{Mon}(+,-)$ be the free monoid on the set $\{+,-\}$.
To any tangle we associate a source $s(\gamma) \in \mathrm{Mon}(\pm):= \mathrm{Mon}(+,-)$ which is a word in $+$ and $-$, that should be read along the oriented interval $I \times \{0\} \times \{1\}$, where we associate the sign $+$ (resp. $-$) to any boundary point of $\gamma$ when the strand at that point is downwards (resp. upwards) oriented. The target $t(\gamma) \in \mathrm{Mon}(\pm)$ of $\gamma$ is defined in a similar way.

We define the strict monoidal \textit{category $\mathcal{T}$ of tangles} by setting its objects to $\mathrm{Mon}(\pm)$ and its morphisms $\mathrm{Hom}_{\mathcal{T}}(s,t)$ to isotopy classes of tangles $\gamma$ with source $s=s(\gamma)$ and target $t = t(\gamma)$. The composition $\gamma \circ \gamma'$ of tangles  $\gamma$ and $\gamma'$ with $t(\gamma')=s(\gamma)$ is obtained by stacking $\gamma'$ on top of $\gamma$. The identity is the trivial tangle with vertical strands, trivial framing and a compatible orientation. The tensor product in $\mathcal{T}$ is defined by the juxtaposition of tangles.

The category of tangles $\mathcal{T}$ contains a subcategory of upwards tangles with objects $\mathrm{Mon(+)} \subset \mathrm{Mon}(\pm)$ and morphisms consting of upwards oriented open tangles. An upwards tangle $\mathcal{K} \in \mathrm{Hom}(+,+)$ is referred to as a \textit{long knot}. A standard closure process exists, resulting in a one-to-one relationship between the subcategories of long knots and closed knots. Studying long knots instead of closed knots has algebraic benefits, which shall be utilized in the upcoming sections.

Instead of working with the conventional generators of morphisms in $\mathcal{T}$ using cups and caps, we will introduce an alternative approach involving a variant of Morse diagrams of tangles. From this perspective, it will be necessary to keep track of the rotational number of the strands within the diagram. A tangle diagram $D$ corresponding to a tangle $\gamma$ is said to be a \textit{rotational tangle diagram} if
\begin{enumerate}[(1)]
\item $\gamma$ is an upwards tangle.
\item all crossings in $D$ are upwards oriented, and all maxima and minima appear in pairs in the following two forms:
\end{enumerate}

\begin{equation*}
\centre{
\centering
\includegraphics[width=0.4\textwidth]{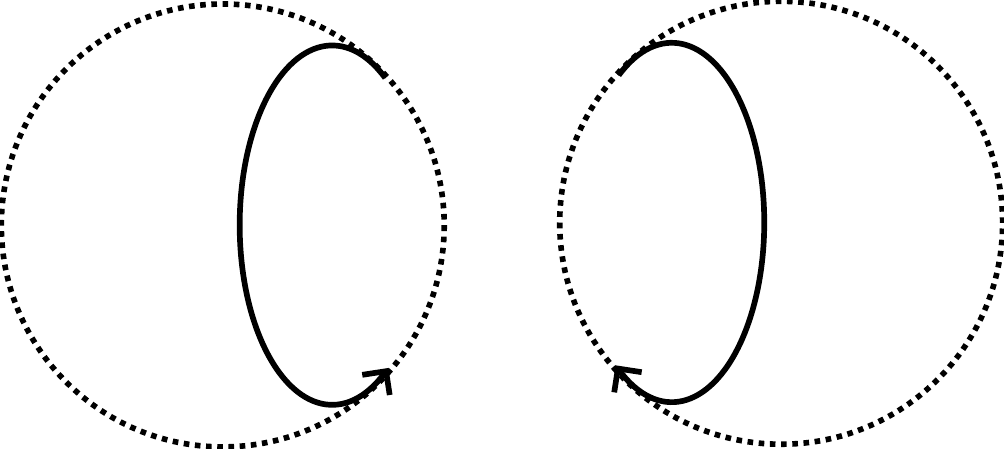}}
\end{equation*}
The rotational diagrams are regarded up to planar isotopy that preserve all maxima and minima. Any tangle diagram can be transformed into a rotational tangle diagram through the application of local planar isotopies, as has been shown in~\cite{becerra2024bar}. In fact, rotational tangle diagrams are related by \textit{rotational Reidemeister moves}, for which we also refer to~\cite{becerra2024bar}.
\begin{proposition}
  \[
\frac{\left\{\begin{array}{c}\textup { upwards tangles } \\ \textup { in }\left(I^{3}\right)\end{array}\right\}}{\textup { isotopy }} =\frac{\left\{\begin{array}{c}\textup { upwards tangle } \\ \textup { diagrams in }I^{2} \end{array}\right\}}{\begin{array}{c}\textup { planar isotopy and} \\ \textup { Reidemeister moves}\end{array}}
=\frac{\left\{\begin{array}{c}\textup {rotational tangles} \\ \textup { in }I^{2} \end{array}\right\}}{\begin{array}{c}\textup { planar isotopy and} \\ \textup { Reidemeister moves}\end{array}}
\]
\end{proposition}
\begin{proof}
See \cite[Cor. 2.3]{becerra2024bar}.
\end{proof}

\subsection{The universal invariant}
\label{sec:theunivinvar}
Let $(A,R,\kappa)$ be a topological ribbon Hopf algebra with unit $1 \in A$ and balancing element $\kappa$. The universal $R$-matrix is given by $R = \sum_i \alpha_i \otimes \beta_i$ with its inverse $R^{-1} = \sum_i \bar{\alpha}_i \otimes \bar{\beta}_i$.

In the following we restrict to the study of open upwards tangles with ordered tangle components. In order to define the universal invariant, an $n$-component open tangle $L= L_1 \cup \dots \cup L_n$ has to be broken apart in terms of cups, caps and crossings. On each strand of these generators, beads are placed, representing elements elements of the ribbon Hopf algebra $A$. In particular, one places ``alpha'' on the overstrand, and ``beta'' on the understrand. The beads are labelled in accordance with the figure below:
\begin{equation}
  \label{eq:beads}
\centre{
\centering
\includegraphics[scale=0.7]{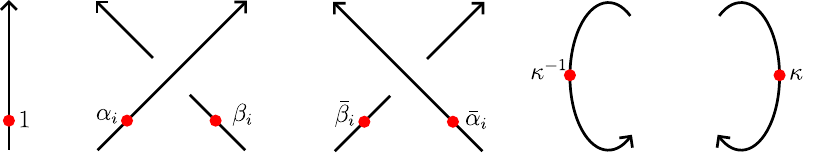}}
\end{equation}
Let be $Z_A(L)_{(k)}$ the product of the labels of the beads from right to left along the orientation of the strand $k$. The \textit{universal invariant} of $L$ associated to the ribbon Hopf algebra $A$ is
\[
  Z_{A}(L):=\sum Z_{A}(L)_{(1)} \otimes \cdots \otimes Z_{A}(L)_{(n)} \in A \otimes \cdots \otimes A
  \]
  where the sum is taken over all the subindices in $R^{\pm 1}$.

  As the name suggests, the universal invariant is preserved under Reidemeister moves and is therefore an invariant of tangles \cite[Thm. 4.5]{ohtsuki2002quantum}.
\begin{example}
  Let $\mathcal{K}$ be the trefoil as illustrated below. Note that we have already decorated the crossings and closures in accordance with the prescription above.
\begin{equation*}
\centre{
\centering
\includegraphics[scale=0.8]{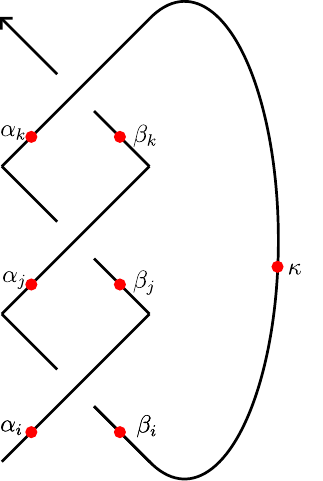}}
\end{equation*}
The universal invariant, associated to the trefoil, is given by
\[
  Z(\mathcal{K})= \sum_{i,j,k}  \beta_k\alpha_j \beta_i\kappa \alpha_k\beta_{j }\alpha_i.
  \]
\end{example}

The universal invariant recovers the Reshetikhin--Turaev invariant. As to see why this is the case,
consider the ribbon category of upwards tangles $\mathcal{T}^{\mathrm{up}}_{\mathrm{Mod_A}}$ by coloring the tangle components with elements of the space of finite-rank left $A$-modules denoted by $\mathrm{Mod}_{A}$. Recall that the Reshetikhin--Turaev functor $RT\colon \mathcal{T}^{\mathrm{up}}_{\mathrm{Mod}_A} \to \mathrm{Mod}_A$ induces an invariant of upwards tangles. Indeed, label each tangle component~$k$ with $V_{k} \in \mathrm{Mod}_A$, together with an action of $A$ given by $\rho_{V_k}$. Write $\sigma$ for the permutation induced by $L$. We have
\begin{align*}
  &R T_{L}\left(V_{1}, \ldots, V_{n}\right) \\
  &=
\sigma_{*} \circ\left(\rho_{V_{1}} \otimes \cdots \otimes \rho_{V_{n}}\right)\left(Z_{A}(L)\right) \in \operatorname{Hom}_{A}\left(V_{1} \otimes \cdots \otimes V_{n}, V_{\sigma(1)} \otimes \cdots \otimes V_{\sigma(n)}\right)
\end{align*}
where $\sigma_{*} \colon V_{1} \otimes \cdots \otimes V_{n} \longrightarrow V_{\sigma(1)} \otimes \cdots \otimes V_{\sigma(n)}$ is induced by $\sigma$, permuting the vector spaces in the tensor factors. Now, as any link  is the closure of a string link, the quantum trace $\mathrm{tr}_q^V(x):= \mathrm{tr}^V(\kappa x)$ relates the universal invariant of $L$ with its closure $\mathrm{cl}(L)$ via the formula
\[
  R T_{\mathrm{cl}(L)}\left( V_{1}, \ldots, V_{n}\right)=\left(\operatorname{tr}_{q}^{V_{1}} \otimes \cdots \otimes \operatorname{tr}_{q}^{V_{n}}\right)\left(Z_{A}(L)\right).
  \]
We note that for modules of infinite rank, it would be impossible to take the trace in the manner described above.

\subsection{XC-algebras}
To construct a knot invariant from the triple $(A, R, \kappa)$ using the outlined procedure, it is not necessary to utilize the full structure of a ribbon Hopf algebra. Instead, Becerra showed  in \cite{becerra2025refinedfunctorialuniversaltangle} that it suffices to consider an algebra structure on $A$ along with specific relations between the elements $R$ and $\kappa$ that ensure the invariance of $Z_A(\mathcal{K})$ under changes of the knot diagram that do not alter the isotopy type of the knot $\mathcal{K}$ it represents. This more general structure, given by the triple $(A, R, \kappa)$, is referred to as an \textit{XC-algebra}.

\begin{definition}
Let $ k $ be a commutative ring with a unit, and let $(A, \mu, 1) $ be a $ k $-algebra. An \emph{$XC $-structure} on $ A $ consists of two invertible elements,
\[
    R \in A \otimes A, \qquad \kappa \in A,
\]
called the \emph{universal $R $-matrix} and the \emph{balancing element}, respectively, which must satisfy the following conditions:
\begin{enumerate}
    \item\label{item:1} $ R^{\pm 1} = (\kappa \otimes \kappa) \cdot R^{\pm 1} \cdot (\kappa^{-1} \otimes \kappa^{-1}) $,
    \item $\mu^{[3]}(R_{31} \cdot \kappa_2) = \mu^{[3]}(R_{13} \cdot \kappa_2^{-1}) $,
    \item $ 1 \otimes \kappa^{-1}= (\mu \otimes \mu^{[3]})(R_{15} \cdot R_{23}^{-1} \cdot \kappa_4^{-1}) $,
    \item $ \kappa \otimes 1 = (\mu^{[3]} \otimes \mu)(R_{15}^{-1} \cdot R_{34} \cdot \kappa_2) $,
    \item $R_{12} R_{13} R_{23} = R_{23} R_{13} R_{12} $,
\end{enumerate}
where $ \mu^{[3]} $ denotes the three-fold multiplication map. For indices $1 \leq i, j \leq n $ with $ i \neq j $, we define
\[
    R_{ij} := \begin{cases}
        (1^{\otimes (i-1)} \otimes \mathrm{Id} \otimes 1^{\otimes (j-i-1)} \otimes \mathrm{Id} \otimes 1^{\otimes (n-j)})(R^{\pm 1}), & i > j, \\
        (1^{\otimes (j-1)} \otimes \mathrm{Id} \otimes 1^{\otimes (i-j-1)} \otimes \mathrm{Id} \otimes 1^{\otimes (n-i)})(\mathrm{flip}_{A,A}R^{\pm 1}), & j > i.
    \end{cases}
\]
Similarly, we write $\kappa_i^{\pm 1} = (1^{\otimes (i-1)} \otimes \mathrm{Id} \otimes 1^{\otimes (n-i)})(\kappa^{\pm 1}) $.

A triple $ (A, R, \kappa) $ consisting of a $k $-algebra equipped with an $ XC $-structure is called an \emph{$XC $-algebra}.
\end{definition}
\begin{proposition}
Every ribbon Hopf algebra $(A,R, \kappa)$ is an $XC$-algebra.
\end{proposition}
\begin{proof}
See \cite[Prop. 4.4]{becerra2025refinedfunctorialuniversaltangle}.
\end{proof}
With the XC-structure in place, we note that the universal invariant is constructed in a manner similar to the description of the previous section, as a consequence of \cite[Thm. 3.10]{becerra2025refinedfunctorialuniversaltangle}. This result will be used to define a ``twisted'' universal invariant.
\subsection{The twisted $R$-matrix}
\label{sec:twisted}
Let $(A,R=\sum_i \alpha_i \otimes \beta_i,\kappa)$ be an $XC$-algebra. It is possible to define a new \textit{twisted} $XC$-algebra with universal $R$-matrix $\check{R}$ that gives rise to the same universal invariant of long knots and satisfies Yang-Baxter equation, as shall be shown in this section. To define this new universal $R$-matrix we introduce the following.
\begin{definition}
  \label{def:conjugating}
  An invertible element $\varphi \in A$ is called a \textit{twisting element} if it satisfies the conditions
  \begin{enumerate}[i)]
    \item $[\varphi,\kappa] = 0$;
    \item  $[\varphi \otimes \varphi, R] = 0$.
  \end{enumerate}
\end{definition}
It is often the case that the balancing element $\kappa$ is itself a twisting element. However, as we shall see, the twisting element of choice will be a slightly deformed version of $\kappa$. For example, in the case of $U_h(\mathfrak{sl}_2)$ over $\mathbb{Q}\llbracket h \rrbracket$, we have $\kappa = \mathrm{Exp}(h H)$, while choosing $\varphi = \mathrm{Exp}(h \mu H)$, for some $\mu \in \mathbb{Q}$.
\begin{definition}
Given a twisting element $\varphi$ and a universal $R$-matrix $R$, the \textit{twisted $R$-matrix} is defined as
\[
\check{R} = (1 \otimes \varphi^{-1}) R  (\varphi \otimes 1)
\]
Alternatively, we denote
\[
  \check{R}_{ij} = \varphi_j^{-1} R_{ij} \varphi_i
 \]
We say that $\check{R}$ is related to $R$ by \textit{twisting}.
\end{definition}
\begin{lemma}
  \label{lem:phiidentity}
$\varphi_j^{-1}R_{ij}\varphi_i = \varphi_i R_{ij} \varphi_j^{-1}$
\end{lemma}
\begin{proof}
This follows from the definition of the twisting element.
\end{proof}

\begin{remark}
Let $P(x \otimes y) = y \otimes x$ be the interchanging map. Note that
\[
P \circ \check{R} = ( \varphi^{-1} \otimes 1) P \circ R (\varphi \otimes 1).
\]
Hence, the actions of $P \circ R$ and $P \circ \check{R}$ on a module $V \otimes V$ differ by a conjugation with $(\varphi^{-1} \otimes 1)$.
\end{remark}
Recall that in the definition of the universal invariant, the universal $R$-matrix $R = \sum_i \alpha_i \otimes \beta_{i}$ and its inverse $R^{-1} = \sum_i \bar{\alpha}_i \otimes \bar{\beta}_i$ are associated to the positive and negative crossings in the following manner.
\begin{equation}
\centre{
\centering
\includegraphics[scale=0.7]{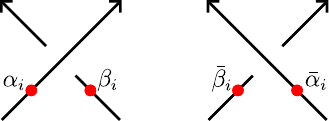}}
\end{equation}
The twisted $R$-matrix is expressed as $\check{R} = \sum_i  \alpha_{i} \varphi \otimes \varphi^{-1} \beta_i $ with inverse equal to $\check{R}^{-1} = \sum_i \varphi^{-1} \bar{\alpha}_i  \otimes \bar{\beta}_{i}\varphi $.
\begin{proposition}
Let \((A,R,\kappa)\) be an \(XC\)-algebra, and let \(\varphi \in A\) be a twisting element. Then \(A\) may be equipped with a new \(XC\)-structure, with universal \(R\)-matrix \(\check{R} \in A \otimes A\) and balancing element \(\kappa\).
\end{proposition}
\begin{proof}
  We need to check that the triple $(A,\check{R},\kappa)$ satisfies the axioms of an $XC$-algebra.

  The first condition follows from the condition $[\varphi, \kappa] = 0$. For the second condition we use Lemma \ref{lem:phiidentity}:
  \begin{align*}
\mu^{[3]}(\check{R}_{31} \cdot \kappa_2) &= \mu^{[3]} \left(\sum_i \varphi^{-1} \beta_i \otimes \kappa \otimes \alpha_i \varphi \right) \\
                                                &= \mu^{[3]} \left(\sum_i \beta_i \varphi^{-1} \otimes \kappa \otimes \varphi \alpha_i\right) \\
    &= \sum_i \beta_i \kappa  \alpha_i \\
    &= \mu^{[3]}(R_{31} \cdot \kappa_2).
  \end{align*}
 Working out $\mu^{[3]}(\check{R}_{13} \cdot \kappa_2^{-1})$ in a similar way results in the desired equality.

 For the third property, we have
  \begin{align*}
    &(\mu \otimes \mu^{[3]})(\check{R}_{15} \cdot \check{R}_{23}^{-1} \cdot \kappa_4^{-1})  \\
    &= (\mu \otimes \mu^{[3]})((\sum_i \alpha_i \varphi \otimes 1 \otimes 1 \otimes 1 \otimes \varphi^{-1} \beta_i)(\sum_j1 \otimes \varphi^{-1} \overline{\alpha}_j \otimes \overline{\beta}_j \varphi \otimes 1 \otimes 1) \kappa_4^{-1}) \\
    &= \sum_{i,j}(\alpha_i \varphi \varphi^{-1} \overline{\alpha}_j \otimes \overline{\beta}_j \varphi \kappa^{-1} \varphi^{-1} \beta_i) \\
    &=  (\mu^{[3]} \otimes \mu)(R_{15}^{-1} \cdot R_{23} \cdot \kappa_4^{-1}) \\
    &=  (1 \otimes \kappa^{-1}).
  \end{align*}
  The fourth property follows anagolously.

  For the last property, we compute
  \begin{align*}
    \check{R}_{12} \check{R}_{13} \check{R}_{23} &= \varphi_2^{-1}R_{12} \varphi_{1}\varphi_3^{-1}R_{13} \varphi_{1}\varphi_3^{-1}R_{23} \varphi_{2} \\
    &=\varphi_{1} R_{12}\varphi_2^{-1} \varphi_3^{-1}R_{13} \varphi_{1}\varphi_{2} R_{23} \varphi_3^{-1} \\
    &=\varphi_{1} R_{12} \varphi_3^{-1}R_{13} \varphi_{1} R_{23} \varphi_3^{-1} \\
    &=\varphi_{1}\varphi_3^{-1} R_{12} R_{13}  R_{23}\varphi_{1} \varphi_3^{-1} \\
    &=\varphi_{1}\varphi_3^{-1} R_{23} R_{13}  R_{12}\varphi_{1} \varphi_3^{-1} \\
    &=\varphi_3^{-1} R_{23} \varphi_2 \varphi_2^{-1} \varphi_{1}R_{13} \varphi_3^{-1} R_{12}\varphi_{1}  \\
    &=\varphi_3^{-1} R_{23} \varphi_2 \varphi_3^{-1} R_{13} \varphi_{1} \varphi_2^{-1} R_{12}\varphi_{1}  \\
    &=\check{R}_{23} \check{R}_{13} \check{R}_{12}.
  \end{align*}
  This proves the proposition.
\end{proof}
With the proposition above, it is possible to define a universal invariant induced by the XC-algebra $(A,\check{R},\kappa)$. For a upwards tangle $\mathcal{T}$, it shall be denoted by $\check{Z}_{A}(\mathcal{T})$. Pictorially, the element $\varphi$ (resp. $\varphi^{-1}$) is represented as a black (resp. white) dot, as illustrated below.
\begin{equation*}
\centre{
\centering
\includegraphics[scale=0.7]{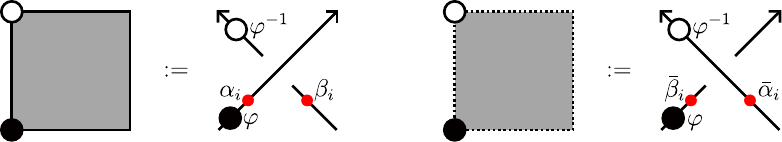}}
\end{equation*}
To ease some of the notation, crossings containing only the ``alpha'' and ``beta'' are replaced with a box, while the elements $\varphi,\varphi^{-1}$ are placed at the corners. From the definition of the twisting element $\varphi$, it is clear that
\[
  \check{R} = \sum_i \alpha_{i} \varphi  \otimes \varphi^{-1} \beta_i  = \sum_i \varphi \alpha_i \otimes  \beta_i \varphi^{-1}.
\]
  In diagrammatic form, this can be expressed as
\begin{equation*}
\centre{
\centering
\includegraphics[scale=0.7]{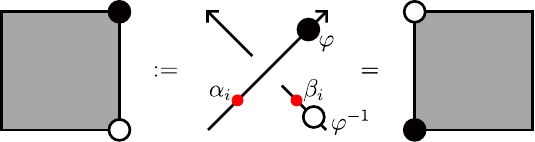}}
\end{equation*}
This identity also holds for the dotted box, associated to $\check{R}^{-1}$.
\begin{lemma}
  \label{lem:movingdots}
We have
\begin{align*}
  (\varphi^{\pm k} \otimes 1) \left( \sum_i \alpha_i \varphi^{n} \otimes \varphi^{-n} \beta_i  \right)&= \left( \sum_i \alpha_i \varphi^{n \pm k} \otimes \varphi^{-(n \pm k)} \beta_i \right) (1 \otimes \varphi^{\pm k} ) ,
\end{align*}
implying
\begin{equation*}
\centre{
\centering
\includegraphics[scale=0.7]{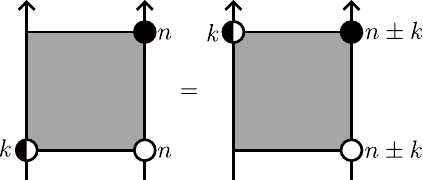}}
\end{equation*}
\noindent where, in the figure, we write a $+$ (resp. $-$) when the dot on the left strand is black (resp. white). A similar expression holds for the the dotted box.
\end{lemma}
\begin{proof}
From the defining relations of the twisting element, we find
  \begin{align*}
    (\varphi^{\pm k} \otimes 1) \left( \sum_i \alpha_i \varphi^n \otimes  \varphi^{-n} \beta_i  \right) &=  \sum_i \varphi^{\pm k} \alpha_{i} \varphi^n \otimes \varphi^{\pm k} \varphi^{-(n \pm k)} \beta_i \\
    &= \left( \sum_i \alpha_i \varphi^{n\pm k} \otimes \varphi^{-(n \pm k)} \beta_{i}\varphi^{\pm k}  \right).
  \end{align*}
  In the expression above, $\alpha_i$ (resp. $\beta_i$) can be replaced with $\bar{\alpha}_i$ (resp. $\bar{\beta}_i$).
\end{proof}
\begin{example}
  \label{ex:braidzcheck}
  Consider the following braid.
\begin{equation*}
  B=
\centre{
\centering
\includegraphics[scale=0.7]{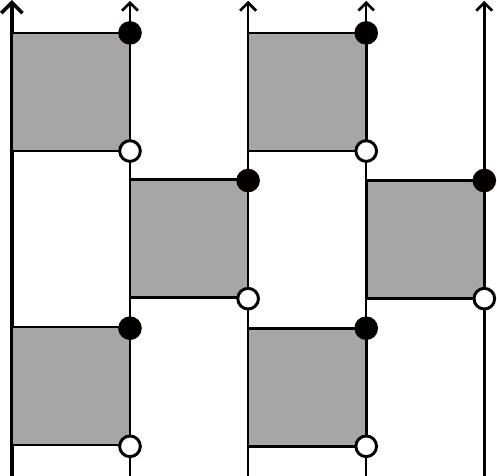}}
\end{equation*}
We associate $\check{Z}_A(B)$ to this braid, as described above. Also, let $\sigma$ be the permuation induces by $B$. Note that there are no elements $\varphi$ or $\varphi^{-1}$ present on the first strand. On the second strand, the first black dot counted from below can be moved along the strand toward the second white dot using Lemma~\ref{lem:movingdots}. These black and white dots then cancel. Due to the application of Lemma~\ref{lem:movingdots}, the dots on the third strand are labeled by $2$, meaning that $\varphi^2$ and $\varphi^{-2}$ lie on the right-hand corners of the gray square adjacent to the second and third strand. By another application of Lemma~\ref{lem:movingdots}, these dots can be moved to the end of the strand, leading to a labeling of $3$ in the fourth strand.  By repeated application of Lemma~\ref{lem:movingdots} we see that $\sigma_{*} \circ \check{Z}_A(B)$ is related to $\sigma_{*} \circ Z_A(B)$ by a conjugation with $1 \otimes \varphi \otimes \dots \otimes \varphi^{4}$. These steps above have been depicted in the following equation.

\begin{align*}
&\centre{
\centering
  \includegraphics[scale=0.55]{example1.pdf}}
  =
\centre{
\centering
  \includegraphics[scale=0.55]{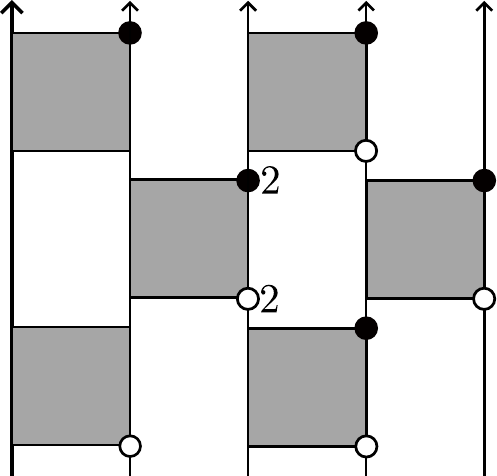}}\\
  &=
\centre{
\centering
  \includegraphics[scale=0.55]{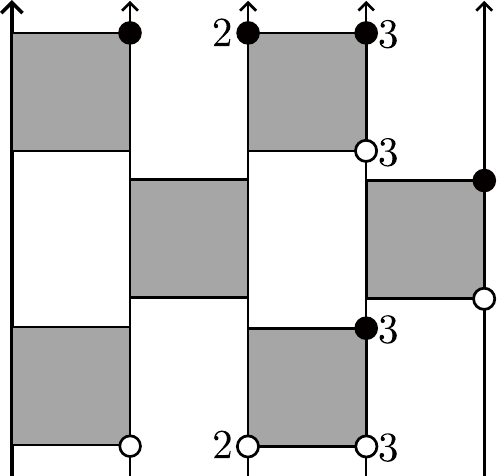}}
  =
\centre{
\centering
  \includegraphics[scale=0.55]{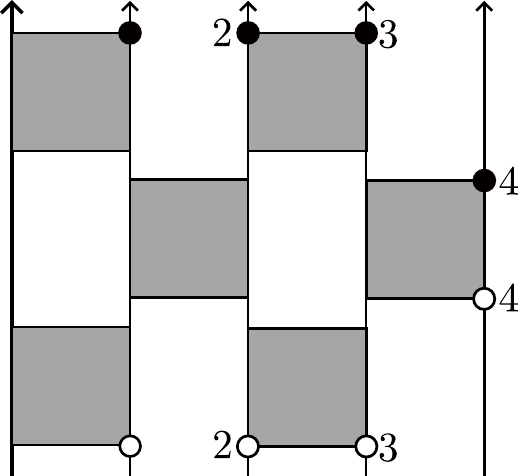}}
\end{align*}
\end{example}

\begin{theorem}
  \label{thm:twistedconj}
    Let $B_N$ be a braid of $N$ strand, and denote by $\sigma$ its induced permutation. Then $\sigma_{*} \circ \check{Z}_A(B_N)$ is related to $\sigma_{*} \circ Z_A(B_N)$ by conjugation with $1 \otimes \varphi \otimes \dots \otimes \varphi^{N-1}$.
\end{theorem}
\begin{proof}
  The steps of Example~\ref{ex:braidzcheck} can be performed for any braid, where one has to perform either of following operations:
  \begin{itemize}
    \item cancel black and white dots that are directly next to another;
    \item apply Lemma~\ref{lem:movingdots} to move the black dot along the strand toward the white dot;
    \item apply Lemma~\ref{lem:movingdots} to move the black (resp. white) dot up (resp. down) toward the end of the strand.
  \end{itemize}
  This way, one ends up with zero dots on the first strand, one white (resp. black) dot at the lower (resp. upper) end of the second strand, etc.
\end{proof}

\begin{corollary}
  \label{cor:equalknotinvariants}
Let $\mathcal{K}$ be a long knot. Then $Z_A(\mathcal{K}) = \check{Z}_A(\mathcal{K}) $
\end{corollary}
\begin{proof}
  Let $B_N$ be a braid of $N$ strands whose partial closure (leaving one strand open) is equal to the long knot $\mathcal{K}$. Also, let $\sigma$ be its induced permutation. By Theorem~\ref{thm:twistedconj}, we know that $ \sigma^{*} \circ \check{Z}_A(B_N)$ is related to $\sigma^{*} \circ Z_A(B_N)$ by conjugation with $1 \otimes \varphi \otimes \dots \otimes \varphi^{N-1}$. The invariant of the closure can be computed through multiplication of the appropriate tensor factors with an element $\kappa$ in between. Since $[\varphi,\kappa] = 0$, we see that the elements $\varphi^k$ and $\varphi^{-k}$ cancel in every closure. Thus, $Z_A(\mathcal{K}) = \check{Z}_A(\mathcal{K}) $.
\end{proof}

\section{Perturbed knot invariants}
\label{sec:perturbedknot}
In this section, we will give a short overview of the algebra $U_h(\mathfrak{sl}_2)$ together with its Verma modules, following \cite{kassel}. Next, we will demonstrate how the universal $R$-matrix of the Drinfeld double $\mathfrak{D}_{\mathfrak{sl}_2}$ can be transformed to match the $R$-matrix of $U_h(\mathfrak{sl}_2)$ through twisting. By establishing an isomorphism between $\mathfrak{D}_{\mathfrak{sl}_2}$ and the algebra $\mathbb{D}$, we can then compare different perturbative expansions.

\subsection{A review of $U_q(\mathfrak{sl}_2)$ and the colored Jones polynomial}
\label{sec:uqsl2}

Consider the 3-dimensional vector space $\mathfrak{sl}_2$ over the field $\mathbb{C}$ with preferred basis $H,E,F$ and Lie bracket $[H,E]=2E$, $[H,F]=-2F$ and $[E,F]=H$. In terms of matrices, the Lie algebra is represented as
\[
E = \begin{pmatrix}
	0 & 1 \\
	0 & 0
\end{pmatrix},
\quad
F = \begin{pmatrix}
	0 & 0 \\
	1 & 0
\end{pmatrix}, \quad
H = \begin{pmatrix}
	1 & 0 \\
	0 & -1
\end{pmatrix}.
\]
It is standard to extend this Lie algebra to the \textit{universal enveloping algebra} denoted by $\mathbf{U} := U(\mathfrak{sl}_2)$. We hereby upgrade the Lie algebra to an associative algebra over $\mathbb{C}$ with the same generators and relations as $\mathfrak{sl}_2$.

Let $k$ be a field of characteristic zero. The \textit{quantum group} $U_q := U_q(\mathfrak{sl}_2)$ is the algebra over $k(q^{1/2})$ with generator $E,F,K,K^{-1}$ and relations
\[
  \begin{array}{c}K K^{-1}=K^{-1} K=1 \\ K E K^{-1}=q^{2} E, \quad K F K^{-1}=q^{-2} F, \\ {[E, F]=\frac{K-K^{-1}}{q-q^{-1}}}.\end{array}
  \]
Set $[n]_{q}:= \frac{q^{n}-q^{-n}}{q-q}$. The algebra is endowed with the structure of a Hopf algebra:
\[ \begin{array}{lll}%
	\Delta(E)=1 \otimes E + E \otimes K  &  \epsilon(E)=0  & S(E)=-EK^{-1} \\
	\Delta(F) = K^{-1} \otimes F + F \otimes 1  &  \epsilon(F)=0 & S(F)=-KF  \\
	\Delta(K)=K \otimes K  &  \epsilon(K)=0  &  S(K) = K^{-1}.
\end{array}\]

To define a universal $R$-matrix, one commonly runs into the issue of convergence. This can be solved by considering the representation theory of $U_q$, or by considering the topological algebra $U_h := U_h(\mathfrak{sl}_2)$ over $k \llbracket h \rrbracket$. This topological algebra is generate by $E,F,H$ with relations
\begin{align*}
  [H,E] = 2E \qquad [H,F] = -2F \\
  [E,F] = \frac{q^H - q^{-H}}{q-q^{-1}} \qquad q = e^h.
\end{align*}
It can be endowed with the structure of a Hopf algebra
\[ \begin{array}{lll}%
	\Delta(E)=1 \otimes E + E \otimes q^H  &  \epsilon(E)=0  & S(E)=-Eq^{-H} \\
	\Delta(F) = q^{-H} \otimes F + F \otimes 1  &  \epsilon(F)=0 & S(F)=-q^HF  \\
	\Delta(H)=H \otimes 1 + 1 \otimes H  &  \epsilon(H)=0  &  S(H) = -H.
\end{array}\]
Moreover, the algebra comes with an universal $R$-matrix
  \begin{align}
    \label{eq:rmatrix}
     R &=\sum_{n=0}^{\infty} \frac{(q - q^{-1})^{n}}{[n]_q !} q^{\frac{H \otimes H}{2} + \frac{n(n-1)}{2}} E^{n} \otimes F^{n} \\
    &= \sum_{n=0}^{\infty} \frac{(q - q^{-1})^{n}}{[n]_q !}  q^{\frac{n(n-1)}{2}} (E q^{-H})^{n} \otimes (q^{H}F)^{n}q^{\frac{H \otimes H}{2}}, \nonumber
  \end{align}
  with inverse
\begin{align*}
  R^{-1}&=\sum_{n=0}^{\infty} \frac{(q^{-1}-q)^{n}}{[n]_q !}   q^{-\frac{n(n-1)}{2}} E^{n} \otimes F^{n} q^{-\frac{H \otimes H}{2}}.
  \end{align*}
	Furthermore, the ribbon element is $v:=K^{-1} u$ where $u= \sum S(\beta) \alpha$ if $R=\sum \alpha \otimes \beta$. This leads to the following.

\begin{theorem}
  \label{thm:ribbonhopf}
The triple $(U_h(\mathfrak{sl}_2),R,v)$ forms a topological ribbon Hopf algebra, with balancing element $q^H$.
\end{theorem}
\begin{proof}
See \cite[Thm 4.14]{ohtsuki2002quantum}.
\end{proof}
\begin{notation}
  \label{not:iota}
Let $\alpha \in \{\pm 1\}$. Define the map $\iota_{\alpha} \colon U_h \to U_h$
\begin{align*}
E \mapsto q^{-H (\alpha -1)/2} F \qquad F \mapsto E q^{H(\alpha-1)/2}\qquad H \mapsto -H \qquad h \mapsto \alpha h.
\end{align*}
\end{notation}
We note that
\begin{align*}
  (\iota_{-1} \otimes \iota_{-1}) (R) = P \circ R^{-1}  \qquad (\iota_{-1} \otimes \iota_{-1})(R^{-1}) = P\circ  R.
\end{align*}
Let $\mathcal{K}$ be a 0-framed long knot with mirror image $\overline{\mathcal{K}}$. Using to the definition of the universal invariant, together with equation~\eqref{eq:beads}, we can establish the following proposition.
\begin{proposition}
  \label{pro:iotaminus}
We have $\iota_{-1}(Z_{U_h}(\mathcal{K})) = Z_{U_h}(\overline{\mathcal{K}})$.
\end{proposition}
One key advantage of considering long knots over closed knots is demonstrated by the following theorem.
\begin{theorem}
  \label{thm:knotcenter}
   The universal invariant $Z_{U_h}(\mathcal{K})$ lies in the center of $U_h$.
\end{theorem}
\begin{proof}
See \cite[Prop. 8.2]{habiro2006bottom}.
\end{proof}

Fix a scalar $\lambda \neq 0$. It is possible to construct the Verma module of highest weight $\lambda$ of $U_q(\mathfrak{sl}_2)$. As shown in \cite{kassel}, one finds an infinite-dimensional vector space $M_q(\lambda)$ over $k(q^{1/2})$ with basis $\{v_i\}_{i \in \mathbb{N}}$ upon which $U_q(\mathfrak{sl}_2)$ acts by
\begin{align*}
  K v_{p}&=\lambda q^{-2 p} v_{p}, \quad K^{-1} v_{p}=\lambda^{-1} q^{2 p} v_{p}, \\ E v_{p+1}&=\frac{q^{-p} \lambda-q^{p} \lambda^{-1}}{q-q^{-1}} v_{p}, \quad F v_{p}=[p+1]_q v_{p+1} \\
  Ev_0 &= 0.
\end{align*}

Note that $M_q(\lambda)$ cannot be simple when $\lambda$ is of the form $\pm q^{\mathrm{n}}$ where $n$ is a non-negative integer. A rescaling of vectors yields
\begin{align}
  \label{eq:verma}
\begin{split}
  K v_{p}&=\lambda q^{-2 p} v_{p}, \quad K^{-1} v_{p}=\lambda^{-1} q^{2 p} v_{p}, \\ E v_{p+1}&=[p+1]_q v_{p}, \quad F v_{p}=\frac{q^{-p} \lambda-q^{p} \lambda^{-1}}{q-q^{-1}} v_{p+1} \\
  Ev_0 &= 0.
\end{split}
\end{align}

The Verma module can equally well be understood in the $h$-formal sense, in which case it is not a vector space over $k(q^{1/2})$ but instead a module over $k \llbracket h \rrbracket$ with an action of $U_h(\mathfrak{sl}_2)$. This module is denoted by $M_h(\lambda)$, with the following actions
\begin{align}
  \label{eq:vermah}
\begin{split}
  H m_{p}&=(\lambda - 2p) m_{p}, \quad Em_0 = 0 \\ E m_{p+1}&=[p+1]_q m_{p}, \quad F m_{p}=\frac{q^{-p+ \lambda} -q^{p- \lambda}}{q-q^{-1}} m_{p+1}.
\end{split}
\end{align}
The element $m_0 \in M_h(\lambda)$ will be called \textit{the highest weight vector}.

Let $\varepsilon \in \{ \pm 1\}$. Up to isomorphism, there exists a unique simple $U_q$-module of dimension $n+1$ that is generated by a highest weight vector of weight $\varepsilon q^n$. Moreover, it is a quotient of the Verma module $M_q(\varepsilon q^n)$. We denote this module by $V_q(\varepsilon,n)$, with $U_q$ action given by
\[
  K v_p = \varepsilon q^{n-2p} v_p \qquad E v_p = \varepsilon [p]_qv_{p-1} \qquad F v_{p-1} = [n-p-1]_qv_p.
  \]
As with the Verma module, it is possible to construct a $U_h(\mathfrak{sl}_2)$-module in an analogous way, which shall be denoted by $V_h(\varepsilon,n)$.

\begin{remark}
    The universal $R$-matrix for the quantum group $U_q$ is not be well-defined, but there is a well-defined $R$-matrix action on the $U_q$-module $V_q(1,n )\otimes V_q(1,n)$. This is because when the Cartan part $q^{\frac{1}{2}H \otimes H}$ acts on a vector, the variable $q$ can be reintroduced as follows:
\[
q^{\frac{1}{2}H \otimes H} v_i \otimes v_j = q^{\frac{1}{2}(n-2i)(n-2j)} v_i \otimes v_j.
\]
\end{remark}
The following proposition is based on \cite[Prop. 42]{willetts2022unification}.
\begin{proposition}
Let $m \in M_h(\lambda)$ with $\lambda \neq 0$ or $m \in V_h(\varepsilon,n)$. There exists a $\beta \in k \llbracket h \rrbracket$ such that $Z_{U_h}(\mathcal{K}) m = \beta m$.
\end{proposition}
From Theorem~\ref{thm:knotcenter}, we know that $Z_{U_h}(\mathcal{K})$ lies in the center of $U_h$. Let us check the statement for a basis element $m_i \in M_h(\lambda)$. We have
$H Z_{U_h}(\mathcal{K}) m_i =Z_{U_h}(\mathcal{K}) H m_i = (\lambda - 2 i) Z_{U_h}(\mathcal{K})   m_i$. Hence, there exists a $\beta_i \in k \llbracket h \rrbracket$ such that $Z_{U_h} m_i = \beta_i m_i$. With this, we find that $$ E Z_{U_h}(\mathcal{K}) m_{i+1} =  \beta_{i+1}E m_{i+1} = \beta_{i+1} [i+1]_q m_i.$$ On the other hand, we also have
\begin{align*}
 E Z_{U_h}(\mathcal{K}) m_{i+1} = Z_{U_h}(\mathcal{K}) E m_{i+1}  =[i+1]_q Z_{U_h}(\mathcal{K})m_i = \beta_i [i+1]_{q} m_i.
\end{align*}
Thus, we conclude $\beta_{i+1} = \beta_i$. Defining $\beta := \beta_i$ yields the proposition for when $m \in M_h(\lambda)$. The same argument holds for when $m \in V_h(\varepsilon,n)$.

\begin{definition}
    The \textit{colored Jones polynomial} of $\mathcal{K}$ is defined as
\[ [n]_q J_{\mathcal{K}}^{n}(q) :=  \mathrm{tr}_q^{V_{n-1}}RT_{V_{n-1}}(\mathcal{K}) \]
where $V_n := V_q(1,n)$ is the $(n+1)$-dimensional $U_q(\mathfrak{sl}_2)$-module defined above.
\end{definition}
\begin{proposition}
  \label{pro:cjuniversal}
  Let $v_0 \in V_h(1,n-1)$. The following equality holds:
  \[
    J_{\mathcal{K}}^n(e^h)v_0 = Z_{U_h}(\mathcal{K}) v_0.
    \]
\end{proposition}
\begin{proof}
  Since $Z_{U_h}$ lies in the center of $U_h$, there exists a $\lambda \in k \llbracket h \rrbracket$ such that $Z_{U_h}v_i = \lambda v_i$ for all $i = 0,\dots,n-1$. Set $q = e^h$. By definition, we have
  \begin{align*}
    \mathrm{tr}_q^{V_h(1,n-1)}RT_{V_h(1,n-1)}(\mathcal{K}) = \sum_{i=0}^{n-1} q^H Z_{U_h}v_i = \lambda \sum_{i=0}^{n-1} q^{n-1-2i}.
  \end{align*}
  On the other hand, we have
  \begin{align*}
    [n]_q = \frac{q^n-q^{-n}}{q-q^{-1}} = \frac{1}{q} \frac{q^n - q^{-n}}{1-q^{-2}} = \frac{1}{q} (q^n-q^{-n})(1+q^{-2} + q^{-4} +\dots) =\sum_{i=0}^{n-1} q^{n-1-2i}.
  \end{align*}
  Thus $\lambda = J_{\mathcal{K}}^n(e^h)$.
\end{proof}
This suggests the following definition.

\begin{definition}
  Let $m_0 \in M_h(\lambda)$ be the highest weight vector. The \textit{colored Jones function} $\mathbf{J}_{\mathcal{K}}(\lambda,h)$ is defined as
  \[
\mathbf{J}_{\mathcal{K}}(\lambda,h) m_0:= Z_{U_{h}}(\mathcal{K})m_0.
    \]
\end{definition}
The colored Jones function can be viewed as an analytic continuation of the colored Jones polynomial. In \cite{willetts2022unification} this has been related to the ADO polynomials.
\begin{proposition}
  \label{pro:polyfunction}
  The following equality holds:
  \[
  J_{\mathcal{K}}^n(e^h) = \mathbf{J}_{\mathcal{K}}(n-1, h).
    \]
\end{proposition}
\begin{proof}
By setting $\lambda = n-1$ in the module $M_h(\lambda)$ one obtains a module, which contains a simple submodule that is equal $V_h(1,n-1)$. This simple submodule contains the the highest weight vector $m_0$, which is labeled as $v_0 \in V_h(1,n-1)$.
\end{proof}

\subsection{The Drinfeld double $\mathfrak{D}_{\mathfrak{sl}_2}$}
In the previous subsection, we described the standard approach to $U_h(\mathfrak{sl}_2)$. The algebra $U_h(\mathfrak{sl}_2)$ is commonly constucted from the Drinfeld double by quotienting out an ideal. This subsection diverges by retaining this ideal, granting us an additional central element, which shall be of use in the forthcoming sections.

The topological Drinfeld double $\mathfrak{D}_{\mathfrak{sl}_2}$ is a topological algebra over $k\llbracket h \rrbracket$ with generators

\begin{align*}
[H,E]=2E && [\widetilde{H},E]= -2E \\
  [H,F]=-2F && [\widetilde{H},F]= 2F \\
[E,F] = \frac{q^{H} - q^{\widetilde{H}}}{q-q^{-1}} && [H,\widetilde{H}] = 0,
\end{align*}
and an $R$-matrix and balancing element
\begin{align*}
  R_{\mathfrak{D}_{\mathfrak{sl}_2}}&=\sum_{n=0}^{\infty} \frac{(q - q)^{n}}{[n]_{q}!} q^{\frac{-H \otimes \widetilde{H}}{2} + \frac{n(n-1)}{2}} E^{n} \otimes F^{n}, \qquad \kappa = q^{\frac{H}{2} - \frac{\widetilde{H}}{2}}.
  \end{align*}
  Based on the information provided, it is possible to construct a knot invariant, by using the definition of the universal invariant as introduced in Section~\ref{sec:theunivinvar}. As opposed to $U_h(\mathfrak{sl}_2)$, we now have a new central element at our disposal, given by
\[
   c := -\frac{H + \widetilde{H}}{2}.
\]
It will be useful to introduce a new set of generators given by
  \[
H' = H +  c = -\widetilde{H} -c \qquad E'= E \qquad F' = e^{h c}F.
    \]
There exists a projection
\[
 p \colon \mathfrak{D}_{\mathfrak{sl}_{2}} \rightarrow U_{h}(\mathfrak{sl}_{2}),
  \]
  given by
  \begin{align*}
    H' \mapsto \hat{H}, \quad E' \mapsto \hat{E}, \quad F' \mapsto \hat{F}, \quad c \mapsto 0.
  \end{align*}
This gives rise to an isomorphism
\[
  \psi \colon \mathfrak{D}_{\mathfrak{sl}_{2}} / \left\langle c \right\rangle \xrightarrow{\sim} U_{h}(\mathfrak{sl}_{2}).
  \]
Here, we note that $\psi(R_{\mathfrak{D}_{\mathfrak{sl}_{2}}}) = R_{\mathfrak{sl}_{2}}$.
Thus, we say that the universal invariant $Z_{\mathfrak{D}_{\mathfrak{sl}_2}}$ dominates universal quantum $\mathfrak{sl}_2$ invariant, and hence all colored Jones polynomials.

In fact, $U_{h}(\mathfrak{sl}_{2})$ is contained in $\mathfrak{D}_{\mathfrak{sl}_{2}}$ as a subalgebra under the map
\[
  i \colon U_{h}(\mathfrak{sl}_{2}) \to \mathfrak{D}_{\mathfrak{sl}_{2}},
  \]
defined by
\[
\hat{H} \mapsto H', \quad \hat{E} \mapsto E' \quad \hat{F} \mapsto F'.
\]
Define
\begin{align}
  \label{eq:rcheck}
\check{R}_{\mathfrak{D}_{\mathfrak{sl}_2}} := \sum_{n=0}^{\infty} \frac{(q - q^{-1})^{n}}{[n]_{q} !} q^{\frac{H' \otimes H'}{2}  + \frac{n(n-1)}{2}} E'^{n} \otimes F'^{n} = i(R_{\mathfrak{sl}_{2}}).
\end{align}
This element can be related to the universal $\mathfrak{D}_{\mathfrak{sl}_2}$ $R$-matrix, as is illustrated by the following.
\begin{proposition}
  \label{pro:rtwistedconj}
  The following identity holds
\[
 q^{\frac{1} 2 c \otimes c} q^{\frac1 2 c \otimes H'} R_{\mathfrak{D}_{\mathfrak{sl}_2}} q^{- \frac1 2 H' \otimes c}= \check{R}_{\mathfrak{D}_{\mathfrak{sl}_2}}.
 \]
\end{proposition}
\begin{proof}
  Note that $q^{\frac1 2 H' \otimes c}(E \otimes F) q^{-\frac1 2 H' \otimes c} = (E\otimes e^{h c} F)$. From this, we find
  \begin{align*}
   & R_{\mathfrak{D}_{\mathfrak{sl}_2}} \\
    &=\sum_{n=0}^{\infty} \frac{(q - q^{-1})^{n}}{[n]_{q} !} q^{\frac{H \otimes (H+2c)}{2} + \frac{n(n-1)}{2}} E^{n} \otimes F^{n} \\
    &=\sum_{n=0}^{\infty} \frac{(q - q^{-1})^{n}}{[n]_{q} !} q^{\frac{(H'-  c) \otimes (H'+c)}{2} + \frac{n(n-1)}{2}} E^{n} \otimes F^{n} \\
 &= q^{-\frac1 2 c \otimes c}\sum_{n=0}^{\infty} \frac{(q - q^{-1})^{n}}{[n]_{q} !} q^{\frac{H' \otimes H'}{2} +\frac1 2 H'\otimes c -\frac1 2  c \otimes H' + \frac{n(n-1)}{2}} E^{n} \otimes F^{n} \\
 &= q^{-\frac1 2 c \otimes c} q^{ -\frac1 2   c \otimes H'} \left( \sum_{n=0}^{\infty} \frac{(q - q^{-1})^{n}}{[n]_q !} q^{\frac{H' \otimes H'}{2}  + \frac{n(n-1)}{2}} E^{n} \otimes (e^{h c}F)^{n} \right) q^{\frac1 2 H' \otimes c}. \qedhere
  \end{align*}
\end{proof}
\begin{notation}
We extend the map $\iota_{\alpha} \colon U_h \to U_h$ from Notation~\ref{not:iota} to a map $\tilde{\iota}_{\alpha} \colon \mathfrak{D}_{\mathfrak{sl}_2} \to \mathfrak{D}_{\mathfrak{sl}_2}$ given by
\[
E' \mapsto q^{-H' (\alpha -1)/2} F' \qquad F' \mapsto E' q^{H'(\alpha-1)/2}\qquad H' \mapsto -H' \qquad h \mapsto \alpha h \qquad c \mapsto c.
  \]
\end{notation}
Given the fact that $\mathfrak{D}_{\mathfrak{sl}_{2}}$ contains $U_{h}(\mathfrak{sl}_{2})$ as  subalgebra, it is rather straightforward to set-up a simple module for $\mathfrak{D}_{\mathfrak{sl}_2}$, where $c$ acts as a scalar. This is done by following the approach of Section~\ref{sec:uqsl2}.
\begin{proposition}
  \label{pro:Dmodule}
 Let $ \mu \in k \llbracket h \rrbracket$. There exists a $\mathfrak{D}_{\mathfrak{sl}_2}$-module, denoted by $\widetilde{M}_h(\lambda,\mu)$, with basis $\{m_i\}_{i \in \mathbb{N}}$, upon which $\mathfrak{D}_{\mathfrak{sl}_2}$ acts by
\begin{align*}
  H' m_{p}&=(\lambda-2p) m_{p}, \quad  c m_p = \mu m_p \\ E' m_{p+1}&=[p+1]_qm_{p}, \quad F' m_{p}= \frac{q^{-p+\lambda} -q^{p- \lambda} }{q-q^{-1}}  m_{p+1} \\
  E' m_0 &= 0.
\end{align*}
Moreover, the simple $U_q$-module of dimension $n+1$, denoted by $V_n$, can be upgraded to a $\mathfrak{D}_{\mathfrak{sl}_2}$-module spanned by the same vector $v_0,\dots,v_{n}$, , upon which $\mathfrak{D}_{\mathfrak{sl}_2}$ acts by
\begin{align*}
  H' v_p = (n-2p) v_p \qquad E' v_p = [p]_{q} v_{p-1} \qquad F' v_{p-1} = [n-p-1]_q v_p \qquad c v_{p} = \mu v_p.
\end{align*}
  This module is denoted by $\widetilde{V}_h(n,\mu)$.
\end{proposition}
Define the knot invariant $\check{Z}_{\mathfrak{D}_{\mathfrak{sl}_2}}$ using the universal $R$-matrix $\check{R}$ as in \eqref{eq:rcheck}. From the fact that $\mathfrak{D}_{\mathfrak{sl}_2}$ contains $U_h(\mathfrak{sl}_2)$ as a subalgebra, it is clear that the universal $\mathfrak{D}_{\mathfrak{sl}_2}$ invariant, acting on a highest weight vector of $\widetilde{M}_h(\lambda,\mu)$, reduces to the colored Jones function. This is exemplified by the following proposition.
\begin{proposition}
  \label{pro:JKZD}
    Let $\mathcal{K}$ be a $0$-framed long knot and let $m_0 \in \widetilde{M}_h(\lambda,\mu)$ be the highest weight vector. The following equalities hold:
  \[
    \mathbf{J}_{\mathcal{K}}(\lambda,e^h) m_0 = \check{Z}_{\mathfrak{D}_{\mathfrak{sl}_2}}(\mathcal{K}) m_{0} = Z_{\mathfrak{D}_{\mathfrak{sl}_2}}(\mathcal{K}) m_{0}.
    \]
\end{proposition}
\begin{proof}
  Recall the homomorphism $i \colon U_h(\mathfrak{sl}_2) \to \mathfrak{D}_{\mathfrak{sl}_{2}}$. Since $\check{R}_{\mathfrak{D}_{\mathfrak{sl}_2}} = i(R_{\mathfrak{sl}_{2}})$ and $\kappa =i(q^{\hat{H}}) $, the first equality follows from Proposition~\ref{pro:polyfunction}.

  Note that $q^{\frac{1}{2} c \otimes H'} m_i = q^{\frac{1}{2} \mu \otimes H'} m_i$ and  $q^{- \frac{1}{2} H' \otimes c} m_i = q^{-\frac{1}{2} H' \otimes \mu} m_i$ for any $i \in \mathbb{N}$.
We write $\varphi = q^{\frac1 2 H \mu}$. This satisfies the definition of a twisting element.
  Thus, in the expression $\check{Z}_{\mathfrak{D}_{\mathfrak{sl}_2}}(\mathcal{K}) m_0$, the $R$-matrix $\check{R}_{\mathfrak{D}_{\mathfrak{sl}_2}}$ is related to $R_{\mathfrak{D}_{\mathfrak{sl}_2}}$ by twisting and a factor $q^{\frac1 2 \mu^2}$. The factor vanishes due to the fact that $\mathcal{K}$ is a 0-framed knot. The second equality now follows with the aid of Corollary~\ref{cor:equalknotinvariants} and Theorem~\ref{thm:knotcenter}.
\end{proof}
\subsection{The algebra $\mathbb{D}$}
Both the algebras $U_h$ and $\mathfrak{D}_{\mathfrak{sl}_2}$ are far too large to allow for a computational efficient universal invariant of knots. To simplify quantum invariants without passing to representations, it was proposed in \cite{bar2021perturbed} to simplify the underlying algebra itself. This can be achieved by setting the field $k = \mathbb{Q} (\epsilon)$ and performing a change of basis. The introduction of a cut-off variable $\epsilon$ allows us to perform computations up to a specified order in $\epsilon$.
\begin{notation}
Let us define the \textit{modified quantum integers} as
  \[
    \{n\}_{q}:=\frac{1-q^{n}}{1-q}.
    \]
    It is readily shown that these modified quantum integers are related to the quantum integer $[n]_q$ via
    \[
      [n]_{q} = q^{-n +1}\{n\}_{q^2}.
      \]
      Moreover, set
      \[
        q_{\epsilon} = \exp(\epsilon h).
        \]
\end{notation}
\begin{definition}
  \label{def:D}
The algebra $\mathbb{D}$ over $\mathbb{Q}(\epsilon)$ is $h$-adically generated by $\mathbf{y},\mathbf{b},\mathbf{a},\mathbf{x}$ with relations
\begin{align*}
&  \mathbf{x y}=q_{\epsilon} \mathbf{y} \mathbf{x}+\frac{\mathbf{1}-\mathbf{A B}}{h} \quad[\mathbf{a}, \mathbf{x}]=\mathbf{x} \quad[\mathbf{b}, \mathbf{x}]=\epsilon \mathbf{x} \\
&  \quad[\mathbf{a}, \mathbf{y}]=-\mathbf{y} \quad[\mathbf{b}, \mathbf{y}]=-\epsilon \mathbf{y} \quad[\mathbf{a}, \mathbf{b}]=0.
\end{align*}
  Alternatively, the algebra can be described using generators $\mathbf{y},\mathbf{t},\mathbf{a},\mathbf{x}$, where we set $\mathbf{t} = \mathbf{b}- \epsilon \mathbf{a}$.
\end{definition}

\begin{theorem}
  The algebra $\mathbb{D}$ can be endowed with a structure of a topological ribbon Hopf algebra over $\mathbb{Q}(\epsilon) \llbracket h \rrbracket $, whereby
\[ \begin{array}{lll}%
	\Delta(\mathbf{x})=1 \otimes \mathbf{x}  +\mathbf{x}\otimes  \mathbf{A}  &  \epsilon(\mathbf{x})=0  & S(\mathbf{x})=-\mathbf{A}^{-1} \mathbf{x} \\
	\Delta(\mathbf{y}) =  1 \otimes \mathbf{y} + \mathbf{y} \otimes \mathbf{B}   &  \epsilon(\mathbf{y})=0 & S(\mathbf{y})=-\mathbf{y} \mathbf{B}^{-1}  \\
	\Delta(\mathbf{a})=\mathbf{a} \otimes 1 + 1 \otimes \mathbf{a}  &  \epsilon(\mathbf{a})=0  &  S(\mathbf{a}) = -\mathbf{a} \\
	\Delta(\mathbf{b})=\mathbf{b} \otimes 1 + 1 \otimes \mathbf{b}  &  \epsilon(\mathbf{a})=0  &  S(\mathbf{b}) = -\mathbf{b},
\end{array}\]
with
\[
    \mathbf{A} := \exp(-\epsilon h \mathbf{a}) \qquad \mathbf{B}:= \exp(-h \mathbf{b}).
  \]
The universal $R$-matrix and the balancing element are given by
  \[
  \mathbf{R}=\sum_{m, n=0}^{\infty} \frac{h^{m+n}}{\{m\}_{q_{\epsilon}} ! n !} \mathbf{y}^{m} \mathbf{b}^{n} \otimes \mathbf{a}^{n} \mathbf{x}^{m} \qquad \mathbf{\kappa}=(\mathbf{A B})^{\frac{1}{2}}
  \]
    where we take the positive square root in the sense of power series in $h$.
\end{theorem}
\begin{proof}
See \cite[Thm. 29]{bar2021perturbed}.
\end{proof}

\begin{remark}
The generators $\mathbf{y},\mathbf{b},\mathbf{a},\mathbf{x}$ spell Yang-Baxter in the universal R-matrix.
\end{remark}
The algebra $\mathbb{D}$ is constructed as a Drinfeld double of a Hopf algebra $\mathbb{B}$ given by the relation $[\mathbf{y},\mathbf{b}] = \epsilon \mathbf{y}$. The construction follows the same approach as the construction of $\mathfrak{D}_{\mathfrak{sl}_2}$ as described in \cite{chari1995guide}. Similar to $\mathfrak{D}_{\mathfrak{sl}_2}$, the algebra $\mathbb{D}$ has the benefit that the center is generated by two central elements instead of one. Alongside the previously mentioned central element $\mathbf{t} \in \mathbb{D}$, the following also holds.
\begin{proposition}
  \label{pro:bigw}
          The element $\mathbf{W}=\mathbf{y} \mathbf{A}^{-1} \mathbf{x}+\frac{q_{\epsilon} \mathbf{A}^{-1}+\mathbf{A T}-\frac{1}{2}(\mathbf{1}+\mathbf{T})(q_{\epsilon}+1)}{h(q_{\epsilon}-1)}$ is central and satisfies $\mathbf{W} = \mathbf{yx} + g( \mathbf{a}, \mathbf{t}) \mod h $ for some power series in $g$.
\end{proposition}
\begin{proof}
See \cite[Thm. 49]{bar2021perturbed}.
\end{proof}

\begin{proposition}
  \label{pro:DDiso}
  The map $\phi \colon \mathbb{D} \to \mathfrak{D}_{\mathfrak{sl}_2}$, given by
\begin{align*}
  \begin{array}{ll} \mathbf{b} \mapsto -\epsilon H/2 & \mathbf{y} \mapsto E     \\ \mathbf{a} \mapsto \widetilde{H}/2 & \mathbf{x} \mapsto \frac{q - q^{-1}}{2h \epsilon^{-1}}q^{-\widetilde{H}} F \\
    h \mapsto 2 h \epsilon^{-1},
  \end{array}
\end{align*}
is an isomorphism of ribbon Hopf $\mathbb{Q}(\epsilon) \llbracket h \rrbracket$-algebras. In particular, this implies that
\begin{align*}
  \phi(\mathbf{t}) &= -\epsilon H/2 - \epsilon \widetilde{H}/2 = \epsilon c \\
  \phi \left( \mathbf{W} \right) &=\left( \frac{\epsilon(q-q^{-1})}{2 h}\right) \left( EF + \frac{q q^{\widetilde{H}} + q^{-1} q^H - \frac{1}{2}(1+q^{H + \widetilde{H}})(q+q^{-1})}{(q-q^{-1})^2}\right).
\end{align*}
\end{proposition}
\begin{proof}
Rather than giving a formal proof of the isomorphism, we observe that the construction of $\mathfrak{D}_{\mathfrak{sl}{2}}$ (see, e.g., Chapter 8 of \cite{chari1995guide}) proceeds through exactly the same steps as the construction of $\mathbb{D}$ (cf. \cite{bar2021perturbed}). The only distinction is that in the case of $U_{h}(\mathfrak{sl}_{2})$ one sets $\epsilon = 1$, whereas in our situation we eliminate $\epsilon$ by inserting factors of $\epsilon^{-1}$ where necessary. The proposition therefore holds by construction.
  \end{proof}
\begin{notation}
The universal invariant arising from the algebra $\mathbb{D}$ and related to a long knot $\mathcal{K}$ shall be denoted by $\mathbf{Z}_{\mathbb{D}}(\mathcal{K})$.
\end{notation}
\begin{corollary}
  \label{cor:JphiD}
Let $m_0$ be the highest weight vector in $\widetilde{M}_h(\lambda,\mu)$. The following equality holds
  \[
    \mathbf{J}_{\mathcal{K}}(\lambda,h) m_0 = \phi(\mathbf{Z}_{\mathbb{D}}(\mathcal{K})) m_0.
    \]
  \end{corollary}
\begin{proof}
This is a direct consequence of Proposition~\ref{pro:JKZD}.
\end{proof}
\begin{remark}
The action of $\phi(Z_{\mathbb{D}}(\mathcal{K}))$ on $m_0 \in \widetilde{M}_h(\lambda,\mu)$ is independent of $\mu$. Thus, we may freely choose any $\mu \in \mathbb{Q}(\epsilon)$, and still obtain the same knot invariant.
\end{remark}
In the previous sections, we emphasized the practicality of the availability of two generators of the center instead of just one. This mainly arises from the way in which the knot invariant $\mathbf{Z}_{\mathbb{D}}(\mathcal{K})$ can be expanded, as will become clear in Theorem~\ref{thm:barnatanveen}. To relate the terms of this expansion to the terms of the large-color expansion, it shall be necessary to compute the action of the central elements $\phi(\mathbf{W})$ and $\phi(\mathbf{T})$ on the highest weight vector of the module $\widetilde{M}_h(\lambda,\mu)$.

Let $\sigma \in \{\pm 1\}$.
\begin{proposition}
  \label{pro:phiwphit}
  Let $m_0 \in \widetilde{M}_h(\lambda,\sigma \lambda)$ be the highest weight vector. We have
  \begin{align*}
    \phi (\mathbf{W})m_0  &=  \frac{ \sigma \epsilon}{4h}  (1-q^{- 2\sigma  \lambda}) m_{0} \\
 \phi(\mathbf{T}) m_0 &= q^{- 2  \sigma \lambda} m_{0}.
  \end{align*}
  Moreover,
  \begin{align*}
    (\tilde{\iota}_{\alpha} \circ \phi)\left(  \mathbf{W} \right)m_0 &=\frac{\sigma \alpha \epsilon}{4 h}(1-q^{-2 \alpha \sigma \lambda}) m_0
 \\
   (\tilde{\iota}_{\alpha} \circ \phi)(\mathbf{T}) m_{0} &= q^{-2 \alpha \sigma \lambda} m_0.
  \end{align*}
\end{proposition}
\begin{proof}
  By using the module structure as given in Proposition~\ref{pro:Dmodule}, together with the image of $\mathbf{W}$ as given in Proposition~\ref{pro:DDiso}, we find
  \begin{align*}
    \phi (\mathbf{W} )m_0 &=\frac{\epsilon}{2h} \frac{q^{\sigma} + q^{- \sigma (2\lambda+1)}- \frac1 2 (1+q^{-2 \sigma \lambda})(q+q^{-1})}{q-q^{-1}} m_0 \\
                                            &=   \frac{ \sigma \epsilon}{4h}  (1-q^{- 2 \sigma \lambda}) m_{0}.
  \end{align*}
  This proves the first identity. The second identity is obvious.

  For the third identity, we note that
    \begin{align*}
     & (\iota_{\alpha}\circ \phi) (\mathbf{W} )m_0 \\
     &= \left( \frac{\epsilon(q-q^{-1})}{2 h}\right)  \Bigg( q^{-\alpha c}E'F' \\
        &+\frac{q^{\alpha} q^{- \alpha H'-\alpha c} + q^{-\alpha} q^{\alpha H'-\alpha c} - \frac{1}{2}(1+q^{- 2 \alpha c})(q+q^{-1})}{(q-q^{-1})^2} \Bigg)m_{0} \\
     &= \left( \frac{\epsilon(q-q^{-1})}{2 h}\right) q^{-\alpha c} \Bigg(   \frac{q^{H'}-q^{-H'}}{q-q^{-1}}  \\
      &+\frac{q^{\alpha} q^{- \alpha H'} + q^{-\alpha} q^{\alpha H'} - \frac{1}{2}(q^{\alpha c}+q^{-  \alpha c})(q+q^{-1})}{(q-q^{-1})^2} \Bigg)m_{0} \\
      &= \frac{\sigma \alpha \epsilon}{4 h}(1-q^{-2 \alpha \sigma \lambda}) m_0.
    \end{align*}
The final equality is also obvious.
\end{proof}

\subsection{The large-color expansion}

After studying Melvin--Morton's work on the expansion of the colored Jones polynomials \cite{melvin1995coloured}, Rozansky conjectured and then showed the existence of a rational expansion of these polynomials with denominators being powers of the Alexander polynomial \cite{rozansky1997higher,rozansky1998universalr}.

 Let $\mathcal{K}$ be a 0-framed long knot.
\begin{theorem}[The large-color expansion]
  \label{thm:rozoverbay}
There exist symmetric Laurent polynomials
  \[
    P^{\mathcal{K}}_i \in \mathbb{Z}[t,t^{-1}], \qquad i \geq 0
    \]
    such that
    \[
      J_{\mathcal{K}}^{n}(q)=\sum_{i=0}^{\infty} \frac{P^{\mathcal{K}}_{i}\left(q^{2n}\right)}{\Delta_{\mathcal{K}}^{2 i+1}\left(q^{2n}\right)}(q^2-1)^{i} \in \mathbb{Q} \llbracket q^2-1 \rrbracket,
      \]
      where $P_{\mathcal{K}}^0(t) = 1$.
\end{theorem}

Whereas Rozansky has shown the existence of the polynomials $P^{\mathcal{K}}_i$,  explicit expressions of $P_1^{\mathcal{K}}$ and $P_2^{\mathcal{K}}$ for knots up to 10 crossings were found by Overbay \cite{overbay2013perturbative}. On the other hand, in \cite{bar2021perturbed} a new approach to universal quantum knot invariants was developed, emphasizing the role of generating functions instead of generators and relations. With the aid of Hopf algebra techniques, the authors expanded the universal invariant $\mathbf{Z}_{\mathbb{D}}(\mathcal{K})$ in terms of $\epsilon$.

\begin{theorem}
  \label{thm:barnatanveen}
  Let $\mathbf{T} = e^{-h \mathbf{t}}$. There exist polynomials $\rho_{k,j}^{\mathcal{K}} \in \mathbb{Z}[\mathbf{T},\mathbf{T}^{-1}]$ such that
  \[
    \mathbf{Z}_{\mathbb{D}}(\mathcal{K})=\frac{1}{\Delta_{\mathcal{K}}} \exp \left(\sum_{k=1}^{\infty} \epsilon^{k} h^k \sum_{j=0}^{2 k} \rho_{k, j}^{\mathcal{K}} \frac{h^{j}\mathbf{W}^{j}}{\Delta_{\mathcal{K}}^{2 k-j}}\right).
    \]
\end{theorem} By considering the Seifert surface of a knot, it is possible to show that the knot invariant $\rho_{1,1}^{\mathcal{K}}$ is proportional to the derivative of the Alexander polynomial. This is summarized in the following proposition. \begin{proposition}
The following equalities hold:
  \begin{enumerate}[\normalfont i)]
    \item $\rho_{1,1}^{\mathcal{K}}(T) = \frac{2  T}{1-T} \dv{}{T} \Delta_{\mathcal{K}}(T)$
    \item $\rho_{1,2}^{\mathcal{K}}(T) = 0$.
  \end{enumerate}
\end{proposition}
\begin{proof}
See \cite[Thm. 53]{bar2021perturbed}.
\end{proof}
In fact, the values for $\rho_{2,1}^{\mathcal{K}}$ and $\rho_{2,2}^{\mathcal{K}}$ have been computed in \cite{becerra2024bar}, which also turn out to be proportional to higher order derivatives of the Alexander polynomial or $\rho_{1,0}$. As in \cite{overbay2013perturbative}, explicit expressions have been given for the polynomials $\rho_{i,0}^{\mathcal{K}}$ in \cite{bar2021perturbed} for knots up to ten crossings.

Due to a difference in convention in the literature, where possibly the mirror image of knots are considered, it will be useful to establish the following.
\begin{proposition}
  \label{pro:mirrorimage}
$\mathbf{J}_{\overline{\mathcal{K}}}(\lambda,h) = \mathbf{J}_{\mathcal{K}}(-\lambda,-h)$.
\end{proposition}
\begin{proof}
  Let $m_0 \in \widetilde{M}_h(\lambda,\sigma \lambda)$ be the highest weight vector. Recall that $\mathbf{J}_{\mathcal{K}}(\lambda,h)m_0= \phi(Z_{\mathbb{D}}(\mathcal{K})) m_0$ conform Proposition \ref{pro:iotaminus}. Moreover, $(\iota_{-1} \circ \phi)(Z_{\mathbb{D}}(\mathcal{K})) m_0 = \mathbf{J}_{\overline{\mathcal{K}}}(\lambda,h)m_0$. The universal invariant $Z_{\mathbb{D}}(\mathcal{K})$ is generated by central elements $\mathbf{T}, \mathbf{W} \in \mathbb{D}$. From Proposition~\ref{pro:phiwphit} it is clear that the eigenvalues of $\phi(\mathbf{W})$ and $\phi(\mathbf{T})$ on a vector $m_0 \in \widetilde{M}_h(\lambda,\lambda)$ are equal to the eigenvalues of $(\iota_{-1} \circ \phi)(\mathbf{W})$ and $(\iota_{-1} \circ \phi)(\mathbf{T})$ on a vector $m_0' \in \widetilde{M}_h(\lambda,-\lambda)$. Since $\iota_{-1}(h) = -h$, we conclude that
 $\mathbf{J}_{\overline{\mathcal{K}}}(\lambda,h) = \mathbf{J}_{\mathcal{K}}(-\lambda,-h)$.
\end{proof}

In \cite{bar2021perturbed}, a relation between the knot polynomials in Theorem~\ref{thm:rozoverbay} and Theorem~\ref{thm:barnatanveen} was conjectured. In the following theorem, we prove this relation for the first-order polynomials.

\begin{theorem}
The polynomials $\rho_{1,0}^{\mathcal{K}}$ and $P^{\mathcal{K}}_1$ are equal.
\end{theorem}
\begin{proof}
To show the equality, we will combine Corollary~\ref{cor:JphiD} and Proposition~\ref{pro:phiwphit}. Let us first write $\mathbf{Z}_{\mathbb{D}}(\mathcal{K})$ into a more convenient form:
  \[
    \mathbf{Z}_{\mathbb{D}}(\mathcal{K})=\frac{1}{\Delta_{\mathcal{K}}} \sum_{k=0}^{\infty} h^{k} \epsilon^{k} \sum_{j=0}^{2 k} \tilde{\rho}^{\mathcal{K}}_{k, j} \frac{h^{j}\mathbf{W}^{j}}{\Delta_{\mathcal{K}}^{2 k-j}}.
    \]
    Note that $\tilde{\rho}_{0,0}^{\mathcal{K}}= 1$, $\rho_{1,0}^{\mathcal{K}} = \tilde{\rho}_{1,0}^{\mathcal{K}}$ and $\rho_{1,1}^{\mathcal{K}} = \tilde{\rho}_{1,1}$. Let $m_0 \in \widetilde{M}_h(\lambda, \sigma \lambda)$ be the highest weight vector, with $\sigma \in \{\pm 1\}$. Also set $T = q^{-2 \sigma \lambda}$. Using Proposition~\ref{pro:phiwphit}, we find
    \begin{align*}
      \mathbf{J}_{\mathcal{K}}(\lambda,e^h)m_0  &= \phi(Z_{\mathbb{D}}(\mathcal{K}))m_0 \\
                                             &= \frac{1}{\Delta_{\mathcal{K}}(T)} \sum_{k=0}^{\infty} (2h)^{k} \sum_{j=0}^{2 k} \tilde{\rho}^{\mathcal{K}}_{k, j}(T) \frac{(2h)^{j}\epsilon^{-j}\phi(\mathbf{W})^{j}}{\Delta_{\mathcal{K}}^{2 k-j}(T)} m_0 \\
      &= \left(\frac{1}{\Delta_{\mathcal{K}}(T)} + \frac{2(\rho_{1,0}(T) + T \sigma  \Delta_{\mathcal{K}}(T) \Delta'_{K})}{\Delta^3_{\mathcal{K}}(T)}h + \dots \right) m_0.
    \end{align*}
    In the last line, the dots represent higher orders in $h$ while keeping $q^{-2 \sigma \lambda}$ fixed. We now set $\sigma =-1$ and $\lambda = n-1$. Using the fact that
    \[
      \frac{1}{\Delta_{\mathcal{K}}(q^{2n -2})} = \frac{1}{\Delta_{\mathcal{K}}(q^{2n})} + \frac{2 q^{2n} \Delta'_{\mathcal{K}}(q^{2n})}{\Delta(q^{2n})^2} h + \dots,
      \]
      we now find
    \begin{align*}
      J_{\mathcal{K}}^n(e^h)m_0 &= \mathbf{J}_{\mathcal{K}}(n-1,e^h)m_0 \\
       &=\left(\frac{1}{\Delta_{\mathcal{K}}(q^{2n})} + \frac{2\rho_{1,0}(q^{2n})}{\Delta^3_{\mathcal{K}}(q^{2n})}h + \dots \right) m_0.
    \end{align*}
Comparing this to Theorem~\ref{thm:rozoverbay} yields the desired expression.
\end{proof}
\section{Experimental verification}
\label{sec:exver}
Bar-Natan and Van der Veen developed a Mathematica program to allow for an efficient computation of the knot invariant $\mathbf{Z}_{\mathbb{D}}$ up to any order in $\epsilon$. This can be found at \url{http://www.rolandvdv.nl/PG/}. We will employ this program alongside Corollary~\ref{cor:JphiD} and Proposition~\ref{pro:phiwphit} to match their expansion in $\epsilon$ with Overbay's computation of the first few orders of the large-color expansion.

Let $m_0 \in \widetilde{M}_h(\lambda,-\lambda)$ be the highest weight vector. Per Proposition~\ref{pro:phiwphit}, we have
\[(
  \tilde{\iota}_1 \circ \phi) (\mathbf{Z}_{\mathbb{D}}(\mathcal{K})) m_0= \mathbf{J}_{\mathcal{K}}(\lambda,h)m_{0}.
\]
   Precisely this identity shall be used to reduce the expansion of $\mathbf{Z}_{\mathbb{D}}(\mathcal{K})$ in terms of $\epsilon$, to the large-color expansion. In the Mathematica program, the output of $\mathbf{Z}_{\mathbb{D}}(\mathcal{K})$ is provided in a ``normal ordered'' form. Thus, we only require the following identities
\begin{align}
  \label{eq:id1}
& (\tilde{\iota}_1 \circ \phi)(\mathbf{x y}) m_0 = 0, \\
  \label{eq:id2}
& (\tilde{\iota}_1 \circ \phi)(\mathbf{a}) m_0 = 0, \\
  \label{eq:id3}
& (\tilde{\iota}_1 \circ \phi)(\mathbf{T}) m_0 = q^{2\lambda} m_0.
\end{align}

Corollary~\ref{cor:JphiD} shall be verified experimentally for the trefoil. Up to second order in $\epsilon$, we find:
\begin{equation*}
\centre{
\centering
\includegraphics[width=\textwidth]{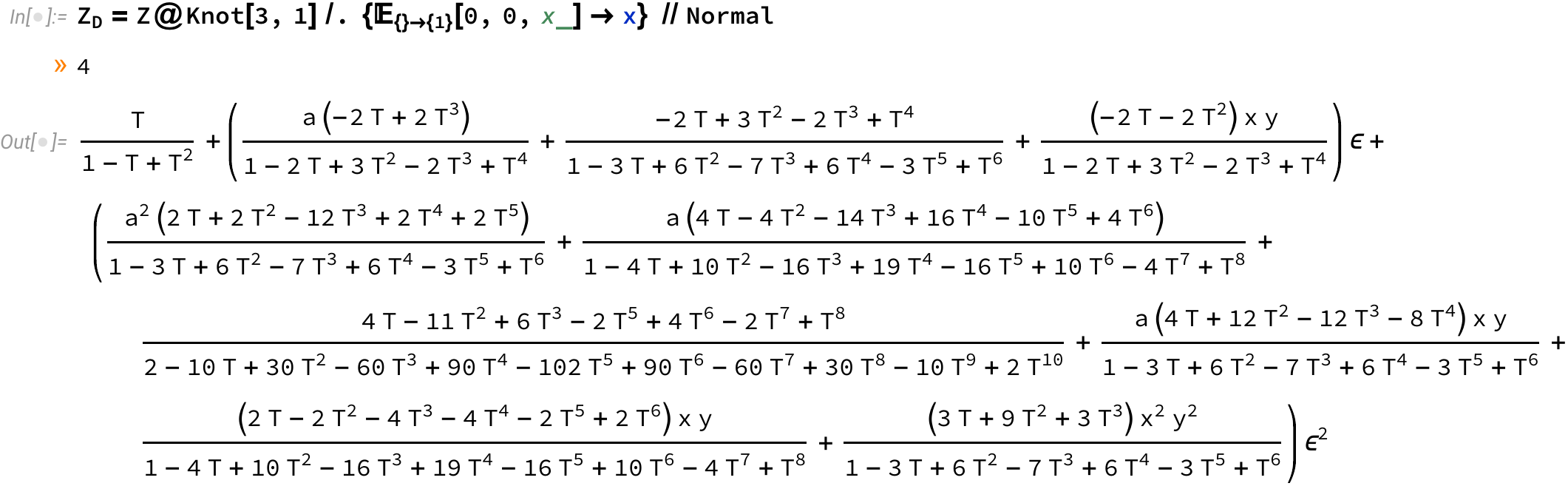}}
\end{equation*}
Now, to compare  the previous expansion colored Jones polynomial, we will set $\mathbf{xy}=0$, $\mathbf{a} =0$ and $\mathbf{T} = T q^2$, consistent with \eqref{eq:id1}, \eqref{eq:id2} and \eqref{eq:id3}. Moreover, note that in the program, the formal parameter $h$ is set to $1$. The map $\phi$ is defined such that $\phi(h) = 2h \varepsilon^{-1}$. To reinsert $h$ back into the expansion, while also applying the map $\phi$, we will set $\epsilon = 2h$.  This leads to:
\begin{equation*}
\centre{
\centering
\includegraphics[width=\textwidth]{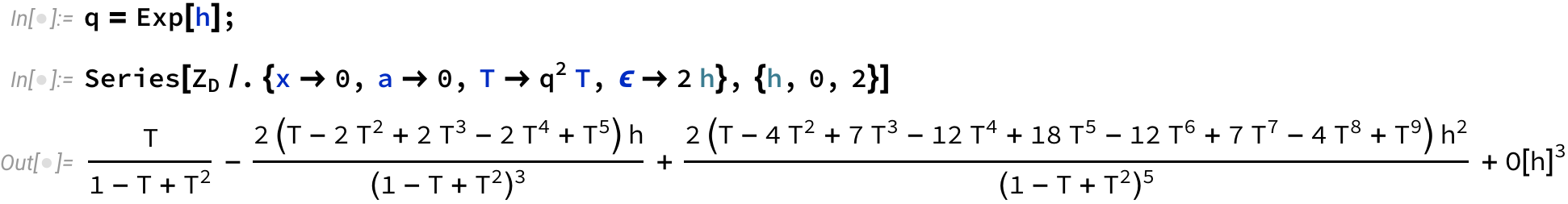}}
\end{equation*}
We can compare this to the result by Overbay \cite{overbay2013perturbative}:
\begin{equation*}
\centre{
\centering
\includegraphics[width=\textwidth]{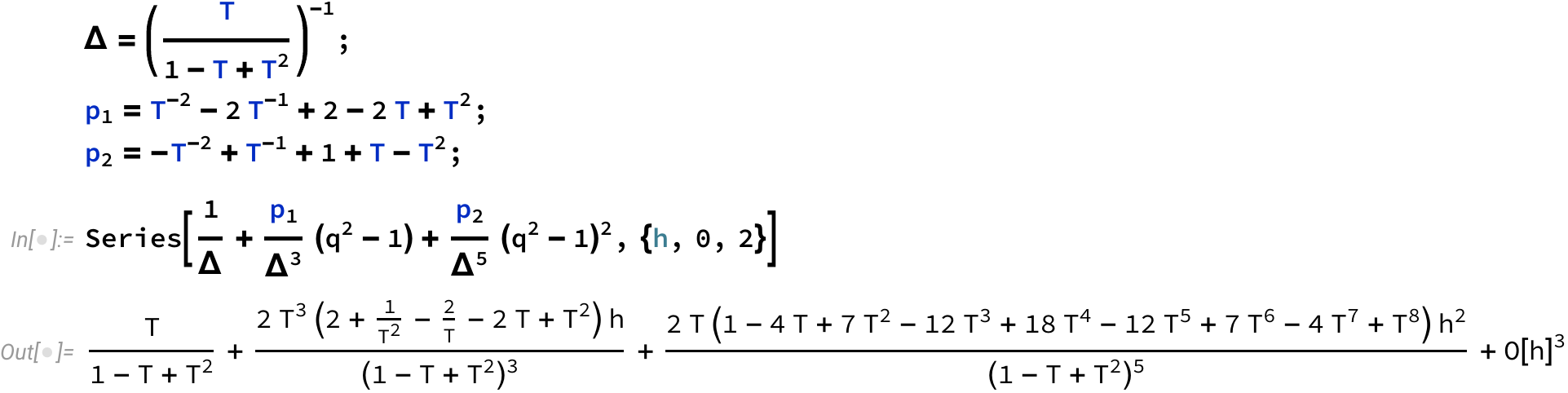}}
\end{equation*}
Notice that both these results match up to a sign difference in the second term. This is due to a difference in convention. Indeed, Overbay uses the mirror image version of the knots that are used in the Mathematica program. With the use of Proposition~\ref{pro:mirrorimage}, this sign difference can be explained.

\appendix


\bibliographystyle{halpha-abbrv}
\bibliography{bibliography}

\begin{thebibliography}{{Bec}24b}
\expandafter\ifx\csname url\endcsname\relax
  \def\url#1{\texttt{#1}}\fi
\expandafter\ifx\csname doi\endcsname\relax
  \def\doi#1{\burlalt{doi:#1}{http://dx.doi.org/#1}}\fi
\expandafter\ifx\csname urlprefix\endcsname\relax\def\urlprefix{URL }\fi
\expandafter\ifx\csname href\endcsname\relax
  \def\href#1#2{#2}\fi
\expandafter\ifx\csname burlalt\endcsname\relax
  \def\burlalt#1#2{\href{#2}{#1}}\fi

\bibitem[Ale28]{Alexander1928TopologicalIO}
J.~W. Alexander.
\newblock Topological invariants of knots and links.
\newblock {\em Transactions of the American Mathematical Society}, 30(2):275,
  1928.
\newblock \doi{10.2307/1989123}.

\bibitem[Bec24a]{becerra2024bar}
J.~Becerra.
\newblock {On Bar-Natan--van der Veen’s perturbed Gaussians}.
\newblock {\em Revista de la Real Academia de Ciencias Exactas, F{\'\i}sicas y
  Naturales. Serie A. Matem{\'a}ticas}, 118(2):46, 2024.
\newblock \doi{10.1007/s13398-023-01536-1}.

\bibitem[{Bec}24b]{thesisbecerra}
J.~{Becerra}.
\newblock {\em Universal quantum knot invariants}.
\newblock PhD thesis, University of Groningen, 2024.
\newblock \doi{10.33612/diss.989716384}.

\bibitem[Bec25]{becerra2025refinedfunctorialuniversaltangle}
J.~Becerra.
\newblock A refined functorial universal tangle invariant, 2025,
  \burlalt{2501.17668}{http://arxiv.org/abs/2501.17668}.
\newblock \urlprefix\url{https://arxiv.org/abs/2501.17668}.

\bibitem[BNG96]{bar1996melvin}
D.~Bar-Natan and S.~Garoufalidis.
\newblock {On the Melvin--Morton--Rozansky conjecture}.
\newblock {\em Inventiones Mathematicae}, 125(1):103--133, 1996.
\newblock \doi{10.1007/s002220050070}.

\bibitem[BV21]{bar2021perturbed}
D.~Bar-Natan and R.~van~der Veen.
\newblock Perturbed {G}aussian generating functions for universal knot
  invariants, 2021.
\newblock \doi{10.48550/ARXIV.2109.02057}.

\bibitem[CP95]{chari1995guide}
V.~Chari and A.~N. Pressley.
\newblock {\em {A guide to quantum groups}}.
\newblock Cambridge university press, 1995.

\bibitem[GM21]{gukov2021two}
S.~Gukov and C.~Manolescu.
\newblock A two-variable series for knot complements.
\newblock {\em Quantum Topology}, 12(1):1–109, 2021.
\newblock \doi{10.4171/qt/145}.

\bibitem[Hab06]{habiro2006bottom}
K.~Habiro.
\newblock Bottom tangles and universal invariants.
\newblock {\em Algebr. Geom. Topol.}, 6:1113--1214, 2006.
\newblock \doi{10.2140/agt.2006.6.1113}.

\bibitem[Jim85]{Jimbo1985AqdifferenceAO}
M.~Jimbo.
\newblock A {$q$}-difference analogue of {$U(\mathfrak{g})$} and the
  {Y}ang-{B}axter equation.
\newblock {\em Lett. Math. Phys.}, 10(1):63--69, 1985.
\newblock \doi{10.1007/BF00704588}.

\bibitem[Jon85]{Jones1985API}
V.~F.~R. Jones.
\newblock {A polynomial invariant for knots via von Neumann algebras}.
\newblock {\em Bulletin of the American Mathematical Society}, 12(1):103–111,
  1985.
\newblock \doi{10.1090/s0273-0979-1985-15304-2}.

\bibitem[Kas95]{kassel}
C.~Kassel.
\newblock {\em Quantum groups}, volume 155 of {\em Graduate Texts in
  Mathematics}.
\newblock Springer-Verlag, New York, 1995.
\newblock \doi{10.1007/978-1-4612-0783-2}.

\bibitem[Law89]{lawrenceuniversal}
R.~Lawrence.
\newblock {\em A universal link invariant using quantum groups}, pages 55--63.
\newblock World Scientific Publishing, 1989.

\bibitem[MM95]{melvin1995coloured}
P.~Melvin and H.~Morton.
\newblock {The coloured Jones function}.
\newblock {\em Communications in Mathematical Physics}, 169:501--520, 1995.
\newblock \doi{10.1007/bf02099310}.

\bibitem[MW24]{martel2024unified}
J.~Martel and S.~Willetts.
\newblock {Unified invariant of knots from homological braid action on Verma
  modules}.
\newblock {\em Proceedings of the London Mathematical Society}, 128, 05 2024.
\newblock \doi{10.1112/plms.12599}.

\bibitem[Oht02]{ohtsuki2002quantum}
T.~Ohtsuki.
\newblock {\em Quantum invariants}, volume~29 of {\em Series on Knots and
  Everything}.
\newblock World Scientific Publishing Co., Inc., River Edge, NJ, 2002.
\newblock A study of knots, 3-manifolds, and their sets.

\bibitem[Ove13]{overbay2013perturbative}
A.~Overbay.
\newblock {Perturbative expansion of the colored Jones polynomial}.
\newblock 2013.

\bibitem[Par20]{park2020large}
S.~Park.
\newblock {Large Color R-Matrix for Knot Complements and Strange Identities}.
\newblock {\em Journal of Knot Theory and Its Ramifications}, 29:2050097, 12
  2020.
\newblock \doi{10.1142/S0218216520500972}.

\bibitem[Roz97]{rozansky1997higher}
L.~Rozansky.
\newblock {Higher order terms in the Melvin-Morton expansion of the colored
  Jones polynomial}.
\newblock {\em Communications in Mathematical Physics}, 183:291--306, 1997.
\newblock \doi{10.1007/BF02506408}.

\bibitem[Roz98]{rozansky1998universalr}
L.~Rozansky.
\newblock {The Universal R-Matrix, Burau Representation, and the
  Melvin–Morton Expansion of the Colored Jones Polynomial}.
\newblock {\em Advances in Mathematics}, 134(1):1--31, 1998.
\newblock \doi{10.1006/aima.1997.1661}.

\bibitem[Wil22]{willetts2022unification}
S.~Willetts.
\newblock {A unification of the ADO and colored Jones polynomials of a knot}.
\newblock {\em Quantum Topology}, 13(1):137--181, 2022.
\newblock \doi{10.4171/QT/161}.

\end{thebibliography}

\end{document}